\documentclass{amsart}

\usepackage{hyperref,amssymb,amsthm,bm,enumerate,mathtools,orcidlink,caption}

\newtheorem{maintheorem}{Theorem}
\newtheorem{theorem}{Theorem}[section]
\newtheorem{proposition}[theorem]{Proposition}
\newtheorem{lemma}[theorem]{Lemma}
\newtheorem{corollary}[theorem]{Corollary}
\newtheorem{conjecture}[theorem]{Conjecture}
\theoremstyle{definition}
\newtheorem{assumption}[theorem]{Assumption}
\newtheorem{problem}[theorem]{Problem}
\newtheorem{note}[theorem]{Note}

\renewcommand{\theenumi}{\roman{enumi}}

\newcommand{\arxiv}[1]{arXiv:\href{https://doi.org/10.48550/arXiv.#1}{#1}}

\title[On the spherical design properties]{On the spherical design properties of \\ a $P$- and $Q$-polynomial association scheme}

\author[J. Lansdown]{Jesse Lansdown\,\orcidlink{0000-0002-8087-1329}}
\address{School of Mathematical and Statistical Sciences, University of Galway, Galway, Ireland}
\email{jesse.lansdown@universityofgalway.ie}
\urladdr{https://www.jesselansdown.com/}
\author[W. J. Martin]{William J. Martin\,\orcidlink{0000-0002-2027-5859}}
\address{Department of Mathematical Sciences, Worcester Polytechnic Institute, Worcester, MA 01609, USA}
\email{martin@wpi.edu}
\author[A. Munemasa]{Akihiro Munemasa\,\orcidlink{0000-0002-2635-0733}}
\address{\href{https://www.math.is.tohoku.ac.jp/}{Research Center for Pure and Applied Mathematics}, Graduate School of Information Sciences, Tohoku University, Sendai 980-8579, Japan}
\address{School of Science, China University of Geosciences, Beijing, China}
\email{munemasa@tohoku.ac.jp}
\author[S. Suda]{Sho Suda\,\orcidlink{0000-0002-5280-2668}}
\address{Department of Mathematics, National Defense Academy of Japan, Yokosuka, Kanagawa 239-8686, Japan}
\email{ssuda@nda.ac.jp}
\author[H. Tanaka]{Hajime Tanaka\,\orcidlink{0000-0002-5958-0375}}
\address{\href{https://www.math.is.tohoku.ac.jp/}{Research Center for Pure and Applied Mathematics}, Graduate School of Information Sciences, Tohoku University, Sendai 980-8579, Japan}
\email{htanaka@tohoku.ac.jp}
\urladdr{https://hajimetanaka.org/}

\date{\today}

\begin{document}

\hypersetup{pdfborder={0 0 0}} 

\begin{abstract}
We show that the strength as a spherical design of the spherical embedding of a $P$- and $Q$-polynomial association scheme with at least three classes with respect to a $Q$-polynomial idempotent is at most five, provided that the multiplicity is at least three.
We also identify the examples that attain this upper bound on the strength.
Our result improves on Suda's earlier upper bound of eight [J. Combin.
Des. 19 (2011)], and is considered dual to the results of Lewis [Discrete Math. 223 (2000)] and Miklavi\v{c} [Electron.
J. Combin. 32 (2025)] concerning the girth of a $Q$-polynomial distance-regular graph with diameter and valency both at least three.
To establish our upper bound, we introduce and discuss a polynomial method that works by constructing an appropriate polynomial that vanishes at every point of the spherical embedding.
\end{abstract}

\maketitle

\hypersetup{pdfborder={0 0 1}} 

\section{Introduction}

For integers $m\geqslant 1$ and $t\geqslant 0$, a nonempty finite subset $\Omega$ of the unit sphere $S^{m-1}$ in $\mathbb{R}^m$ centered at the origin is called a \emph{spherical $t$-design} in $S^{m-1}$ if
\begin{equation*}
    \frac{1}{\omega_{m-1}}\int_{S^{m-1}} f(\bm{\xi})\,\mathrm{d}\sigma(\bm{\xi})=\frac{1}{|\Omega|} \sum_{\bm{\xi}\in\Omega}f(\bm{\xi})
\end{equation*}
for every polynomial $f(\bm{\xi})=f(\xi_1,\xi_2,\dots,\xi_m)$ of degree at most $t$, where the integral in the LHS is the surface integral, and $\omega_{m-1}$ denotes the surface volume of $S^{m-1}$.
Note that a spherical $t$-design is also a spherical $i$-design for $0\leqslant i<t$.
The concept of spherical designs is due to Delsarte, Goethals, and Seidel \cite{DGS1977GD}.
We refer the reader to \cite{BB2009EJC,BBTZ2017GC} for recent updates on spherical designs.

One of the main sources of spherical designs is spherical embeddings of \emph{association schemes}.
Let $(X,\mathcal{R})$ be a symmetric association scheme with $d\geqslant 1$ classes.
(Formal definitions begin in Section~\ref{sec: P- and Q-polynomial association schemes}.)
Let $E=E_i$ $(i>0)$ be one of its primitive idempotents, and let $m=\operatorname{rank}E$, the \emph{multiplicity} of $E$.
Then, $|X|m^{-1}E$ is a real Gram matrix with constant diagonal entries equal to one, so there exist unit vectors $\bm{\xi}_x\in\mathbb{R}^m$ $(x\in X)$ such that $\bm{\xi}_x\cdot\bm{\xi}_y=|X|m^{-1}E_{x,y}$ $(x,y\in X)$, where $\cdot$
 denotes the dot product on $\mathbb{R}^m$.
We call the set $\tilde{X}=\{\bm{\xi}_x:x\in X\}\subset S^{m-1}$ the \emph{spherical embedding} of $(X,\mathcal{R})$ with respect to $E$.
We remark that $\tilde{X}$ is uniquely determined up to orthogonal transformation.
A quick example is the twelve vertices of the regular icosahedron in $S^2$, which form a spherical $5$-design and are the spherical embedding of an association scheme with three classes.
Likewise, after rescaling, the $240$ vectors of the $\mathrm{E}_8$ root system and the $196560$ shortest vectors of the Leech lattice $\Lambda_{24}$ respectively form a spherical $7$-design in $S^7$ and a spherical $11$-design in $S^{23}$, and these also arise in this way.

Let $t(\tilde{X})$ denote the largest integer $t$ such that $\tilde{X}$ is a spherical $t$-design.
We call $t(\tilde{X})$ the (maximum) \emph{strength} of $\tilde{X}$.
In this paper, we focus primarily on the case where $(X,\mathcal{R})$ is $Q$-\emph{polynomial} and $E$ is a corresponding $Q$-polynomial idempotent.
In this case, the strength of $\tilde{X}$ (or more precisely, $t(\tilde{X})+1$) is considered a parameter dual to the girth $g(\Gamma)$ of a \emph{distance-regular} graph $\Gamma$, or equivalently, a $P$-\emph{polynomial} association scheme with a specified $P$-polynomial adjacency matrix.
This is explained, e.g., by Suda's formula for the strength in terms of the Krein parameters \cite[Theorem~3.1]{Suda2011JCD} (cf.~Theorem~\ref{t in terms of Krein numbers}).
For example, $t(\tilde{X})\geqslant 3$ if and only if the Krein parameter $a_1^*=0$, whereas $g(\Gamma)\geqslant 4$ if and only if the intersection number $a_1=0$.
See also \cite[Proposition~4.1]{CGS1978IM} and \cite[Lemma~2]{Munemasa2004EJC}.
Lewis \cite[Corollary~30]{Lewis2000DM} showed that $g(\Gamma)\leqslant 6$ if $\Gamma$ is $Q$-polynomial and has valency (or degree) $k\geqslant 3$, and very recently, Miklavi\v{c} \cite[Theorem~20]{Miklavic2025EJC} determined all such $\Gamma$ with $g(\Gamma)=6$.
Miklavi\v{c}'s classification involves two families: (a) the collinearity graphs of the generalized hexagons of order $(1,k-1)$ $(k\geqslant 3)$, which have fixed diameter $d=3$, and (b) the Odd graphs $O_{d+1}$ with unbounded diameter $d\geqslant 3$.
On the dual side, Suda \cite[Theorem~4.2]{Suda2011JCD} proved the bound $t(\tilde{X})\leqslant 8$ under the equivalent setting that $(X,\mathcal{R})$ is both $P$- and $Q$-polynomial and $E$ is $Q$-polynomial with $m\geqslant 3$.
Our main result improves on Suda's bound and mirrors Lewis' and Miklavi\v{c}'s results:

\begin{maintheorem}\label{main theorem}
Let $(X,\mathcal{R})$ be a $P$- and $Q$-polynomial association scheme with $d\geqslant 3$ classes, and let $E$ be a corresponding $Q$-polynomial idempotent, where we assume that the multiplicity $m\geqslant 3$.
Let $\tilde{X}\subset S^{m-1}$ be the spherical embedding of $(X,\mathcal{R})$ with respect to $E$.
Then, the strength $t(\tilde{X})$ of $\tilde{X}$ satisfies $t(\tilde{X})\leqslant 5$.
Moreover, we have $t(\tilde{X})=5$ if and only if $(X,\mathcal{R})$ is the association scheme of one of the following distance-regular graphs:
\renewcommand{\theenumi}{\alph{enumi}}\begin{enumerate}
\item\label{Taylor} a nonbipartite Taylor graph $(d=3)$ with intersection array $\{k,\mu,1;1,\mu,k\}$, where
\begin{align*}
k=\frac{(m-1)(m+2)}{2}, \qquad
\mu=\frac{m^2+m-4-(m-3)\sqrt{m+2}}{4};
\end{align*}
\item the Hermitian dual polar graph $[^2\!A_{2d-1}(2)]$.
\end{enumerate}
\end{maintheorem}

A nonbipartite \emph{Taylor graph} is a nonbipartite distance-regular antipodal double-cover with diameter $d=3$ of a complete graph.
The \emph{Hermitian dual polar graph} $[^2\!A_{2d-1}(2)]$ has as vertices the maximal (i.e., $d$-dimensional) isotropic subspaces of the $2d$-dimensional vector space $\mathbb{F}_4^{2d}$ over the finite field $\mathbb{F}_4$ with respect to a fixed nondegenerate Hermitian form, where two vertices $x$ and $y$ are adjacent if and only if $\dim x\cap y=d-1$; cf.~\cite[Section~9.4]{BCN1989B}.
That the graph $[^2\!A_{2d-1}(2)]$ gives rise to a spherical $5$-design was observed earlier by Munemasa \cite[Theorem~1]{Munemasa2004EJC}.
It should be remarked that these graphs have precisely two $Q$-polynomial idempotents.
For a nonbipartite Taylor graph $\Gamma$, the primitive idempotents associated with the second largest and the smallest eigenvalues are $Q$-polynomial (see, e.g., \cite[p.~431]{BCN1989B}), and the expressions for the parameters $k$ and $\mu$ above are the conditions for the one associated with the second largest eigenvalue.
However, the distance-$2$ graph $\Gamma_2$ of $\Gamma$ is also a nonbipartite Taylor graph, and the other $Q$-polynomial idempotent is associated with the second largest eigenvalue of $\Gamma_2$.
Hence, Theorem~\ref{main theorem}\,\eqref{Taylor} covers both $Q$-polynomial idempotents of $\Gamma$.
For the graph $[^2\!A_{2d-1}(2)]$, the $Q$-polynomial idempotents are again associated with the second largest and the smallest eigenvalues, but only the one associated with the smallest eigenvalue defines a spherical $5$-design.

The spherical $5$-designs obtained from the nonbipartite Taylor graphs as in Theorem~\ref{main theorem}\,\eqref{Taylor} are \emph{tight}.
Here, for a spherical $t$-design $\Omega$ in $S^{m-1}$, we have
\begin{equation}\label{Fisher bound}
    |\Omega|\geqslant \begin{cases} \dbinom{m+e-1}{e} + \dbinom{m+e-2}{e-1} & \text{if $t=2e$ is even}, \\[2.5mm] 2\dbinom{m+e-1}{e} & \text{if $t=2e+1$ is odd}, \end{cases}
\end{equation}
and $\Omega$ is called tight if equality holds above.
In general, every tight spherical $t$-design is the spherical embedding of a $Q$-polynomial association scheme with $\lceil t/2\rceil$-classes with respect to a $Q$-polynomial idempotent (see, e.g., \cite[p.~380]{DGS1977GD} and \cite[p.~322]{Godsil1993B}), and the tight spherical $5$-designs are, up to orthogonal transformation, shown to be in bijection with the isomorphism classes of nonbipartite Taylor graphs as in Theorem~\ref{main theorem}\,\eqref{Taylor}; see Corollary~\ref{Q-bipartite and d=3 imply P-polynomial}.
By this correspondence, all the known restrictions on the dimension $m$ for a tight spherical $5$-design apply to the graphs in Theorem~\ref{main theorem}\,\eqref{Taylor}.
For example, the tight spherical $5$-designs $\Omega$ in $S^{m-1}$ are antipodal (i.e., $\Omega=-\Omega$) and are in turn in bijection with the sets of equiangular lines in $\mathbb{R}^m$ that meet the Absolute Bound $m(m+1)/2$ by \cite[Theorems~5.12, 6.8]{DGS1977GD}, and hence it follows from \cite[Theorem~3.5]{LS1973JA} that either $m=3$, or $m+2$ is the square of an odd integer. 
See \cite{BMV2005SPMJ,NV2013SPMJ} for further restrictions on $m$.
Currently, examples of tight spherical $5$-designs are known only for $m\in\{3,7,23\}$. 
The vertices of the regular icosahedron give such an example with $m=3$.
We may remark that the two other aforementioned examples of spherical $7$- and $11$-designs are also tight.

The present paper is organized as follows.
Section~\ref{sec: P- and Q-polynomial association schemes} reviews basic facts about $P$- and $Q$-polynomial association schemes.
We recommend \cite{BBIT2021B,BI1984B,BCN1989B,DKT2016EJC,Godsil1993B,MT2009EJC} for more background material.
Section~\ref{sec: vanishing polynomials} introduces and discusses a polynomial method for bounding $t(\tilde{X})$.
This method works by constructing an $m$-variate polynomial that vanishes on $\tilde{X}$ but does not entirely vanish on $S^{m-1}$.
We prove the bound $t(\tilde{X})\leqslant 5$ in Theorem~\ref{main theorem} in this section (see Proposition~\ref{Q-polynomial}).
While it is possible to prove Theorem~\ref{main theorem} solely by the case-by-case approach in Section~\ref{sec: classification of schemes with t=5} and without relying on this method, this section is of independent interest.
In particular, Martin and Steele \cite{MS2015JAC} initiated the study of the ideal of the polynomial ring consisting of the polynomials that vanish on $\tilde{X}$.
Section~\ref{sec: classification of schemes with t=5} completes the proof of Theorem~\ref{main theorem} by classifying the $P$- and $Q$-polynomial association schemes such that $t(\tilde{X})=5$.
We first show that such association schemes are either $Q$-\emph{bipartite} or \emph{almost} $Q$-\emph{bipartite}, and then invoke the known classifications of the corresponding distance-regular graphs due to Curtin \cite{Curtin1998DM}, Nomura \cite{Nomura1995JCTB}, Dickie \cite{Dickie1995D}, and Dickie and Terwilliger \cite{DT1996EJC}.
The almost $Q$-bipartite case with $d=3$ was not fully covered by their (more specifically, Dickie's) results, so we discuss this case at length.
Section~\ref{sec: concluding remarks} concludes the paper with supplementary results and open problems.

\section{$P$- and $Q$-polynomial association schemes}\label{sec: P- and Q-polynomial association schemes}

Let $X$ be a finite set, and let $\mathbb{R}^{X\times X}$ (resp.~$\mathbb{C}^{X\times X}$) be the set of real (resp.~complex) matrices with rows and columns indexed by $X$.
Let $\mathcal{R}=\{R_0,R_1,\dots,R_d\}$ be a set of nonempty subsets of $X\times X$.
For every $i$, let $A_i\in\mathbb{R}^{X\times X}$ be the adjacency matrix of the (di)graph $(X,R_i)$.
The pair $(X,\mathcal{R})$ is called an \emph{association scheme} \emph{with} $d$ \emph{classes} if
\begin{enumerate}[({AS}1)]
\item\label{AS1} $A_0=I$, the identity matrix;
\item\label{AS2} $A_0+A_1+\cdots+A_d=J$, the all ones matrix;
\item $A_i^{\top}\in\{A_0,A_1,\dots,A_d\}$ for $0\leqslant i\leqslant d$;
\item\label{AS4} $A_iA_j$ is a linear combination of $A_0,A_1,\dots,A_d$ for $0\leqslant i,j\leqslant d$.
\end{enumerate}
By (AS\ref{AS1}) and (AS\ref{AS4}), the $\mathbb{C}$-vector space $\bm{A}$ spanned by the $A_i$ is a $\mathbb{C}$-subalgebra of $\mathbb{C}^{X\times X}$, called the \emph{adjacency algebra} of $(X,\mathcal{R})$.
The $A_i$ are linearly independent by (AS\ref{AS2}) and thus form a basis of $\bm{A}$.
Let $p_{i,j}^h\in\mathbb{R}$ $(0\leqslant i,j,h\leqslant d)$ denote the structure constants of $\bm{A}$ with respect to this basis, i.e.,
\begin{equation}\label{intersection numbers}
    A_iA_j=\sum_{h=0}^d p_{i,j}^h A_h \qquad (0\leqslant i,j\leqslant d).
\end{equation}
The $p_{i,j}^h$ are nonnegative integers and are called the \emph{intersection numbers} of $(X,\mathcal{R})$.
We say that $(X,\mathcal{R})$ is \emph{commutative} if the algebra $\bm{A}$ is commutative (i.e., $p_{i,j}^h=p_{j,i}^h$ for all $i,j,h$), and that $(X,\mathcal{R})$ is \emph{symmetric} if the $A_i$ are symmetric matrices.
By applying ${}^{\top}$ to both sides in \eqref{intersection numbers}, we immediately see that a symmetric association scheme is commutative.

From now on, assume that $(X,\mathcal{R})$ is a symmetric association scheme.
An ordering $A_0,A_1,\dots,A_d$ of the adjacency matrices is called $P$-\emph{polynomial} if for all $i,j,h$ $(0\leqslant i,j,h\leqslant d)$, we have $p_{i,j}^h=0$ (resp.~$p_{i,j}^h>0$) whenever one of $i,j,h$ is greater than (resp.~equal to) the sum of the other two.
In this case, we have
\begin{equation}\label{3-term recurrence}
    AA_i=b_{i-1}A_{i-1}+a_iA_i+c_{i+1}A_{i+1} \qquad (0\leqslant i\leqslant d),
\end{equation}
where we set $A=A_1$,
\begin{equation*}
    a_i=p_{1,i}^i, \qquad b_i=p_{1,i+1}^i, \qquad c_i=p_{1,i-1}^i
\end{equation*}
for $0\leqslant i\leqslant d$, and where $b_{-1}A_{-1}=c_{d+1}A_{d+1}:=0$.
This means that the graph $\Gamma=(X,R_1)$ is a \emph{distance-regular} graph with diameter $d$ and that each $A_i$ is the $i^{\mathrm{th}}$ distance matrix of $\Gamma$.
The $P$-polynomial ordering $A_0,A_1,\dots,A_d$ is determined by $A=A_1$, and hence we also say that the adjacency matrix $A$ is $P$-\emph{polynomial}.
Note that $\Gamma$ is regular with valency (or degree) $k=b_0$, and that $c_1=1$, $c_0=b_d=0$, and $a_i=k-b_i-c_i$ $(0\leqslant i\leqslant d)$.
Recall the \emph{intersection array} of $\Gamma$:
\begin{equation}\label{intersection array}
    \{b_0,b_1,\dots,b_{d-1};c_1,c_2,\dots,c_d\}.
\end{equation}
We say that $(X,\mathcal{R})$ is $P$-\emph{polynomial} if it possesses a $P$-polynomial adjacency matrix.

For the rest of this section, we assume that $A_0,A_1,\dots,A_d$ is a $P$-polynomial ordering.
The adjacency matrix $A=A_1$ of $\Gamma$ has exactly $d+1$ distinct eigenvalues $\theta_0,\theta_1,\dots,\theta_d\in\mathbb{R}$ because $A$ generates the adjacency algebra $\bm{A}$ by \eqref{3-term recurrence}.
For $0\leqslant i\leqslant d$, let $E_i\in\mathbb{R}^{X\times X}$ denote the orthogonal projection onto the eigenspace of $A$ associated with $\theta_i$.
Then, $E_0+E_1+\cdots+E_d=I$, $E_iE_j=\delta_{i,j}E_i$ $(0\leqslant i,j\leqslant d)$, and we have
\begin{equation}\label{eigenvalues}
    A=\sum_{i=0}^d \theta_i E_i.
\end{equation}
The $E_i$ are polynomials in $A$ and form another basis of $\bm{A}$.\footnote{We are assuming here that $(X,\mathcal{R})$ is $P$-polynomial, but this is just for ease of exposition. It is known that the adjacency algebra of every commutative association scheme has the basis of the primitive idempotents; see, e.g., \cite[Section~2.3]{BI1984B}.}
We call the $E_i$ the \emph{primitive idempotents} of $(X,\mathcal{R})$.
We always set
\begin{equation}\label{theta0}
    \theta_0:=k=b_0,
\end{equation}
so that
\begin{equation}\label{E0}
    E_0=\frac{1}{|X|}J.
\end{equation}
For $0\leqslant i\leqslant d$, let
\begin{equation*}
    m_i=\operatorname{trace}(E_i)=\operatorname{rank}(E_i).
\end{equation*}
We call $m_i$ the $i^{\mathrm{th}}$ \emph{multiplicity} of $(X,\mathcal{R})$.

Let $E$ be one of the primitive idempotents and let $m$ denote its multiplicity.
Since $E\in\bm{A}$, there exist scalars $\theta_0^*,\theta_1^*,\dots,\theta_d^*\in\mathbb{R}$ such that
\begin{equation}\label{dual eigenvalues}
    E=\frac{1}{|X|}\sum_{i=0}^d \theta_i^* A_i.
\end{equation}
By (AS\ref{AS2}), the $\theta_i^*$ are the (not necessarily distinct) entries of the matrix $|X|E$.
We call the $\theta_i^*$ the \emph{dual eigenvalues} of $(X,\mathcal{R})$ \emph{associated with} $E$.
Observe by (AS\ref{AS1}) that
\begin{equation}\label{theta0*}
    \theta_0^*=m,
\end{equation}
and that the matrix
\begin{equation*}
    \frac{|X|}{m}E
\end{equation*}
is a rank-$m$ positive semidefinite matrix all of whose diagonal entries are equal to one.
Hence, it follows that there exists a set of points $\tilde{X}=\{\bm{\xi}_x:x\in X\}$ on the unit sphere $S^{m-1}\subset\mathbb{R}^m$ such that
\begin{equation*}
    \bm{\xi}_x\cdot \bm{\xi}_y= \sigma_{\partial(x,y)}:=\frac{\theta_{\partial(x,y)}^*}{m} \qquad (x,y\in X),
\end{equation*}
where $\cdot$ denotes the dot product on $\mathbb{R}^m$ and $\partial$ denotes the path-length distance function on $X$ with respect to $\Gamma$.
The set $\tilde{X}$ is unique up to orthogonal transformation.
We call $\tilde{X}$ the \emph{spherical embedding} of $(X,\mathcal{R})$ \emph{with respect to} $E$ and the $\sigma_i$ the \emph{cosines} of $(X,\mathcal{R})$ \emph{associated with} $E$.
Note that $\sigma_0=1$.
We will require that the $\sigma_i$ are distinct, and hence that the $\bm{\xi}_x$ are distinct.
It should be mentioned that this condition is quite weak.
When $k\geqslant 3$, if we arrange the eigenvalues in decreasing order $\theta_0>\theta_1>\cdots>\theta_d$, then the $\sigma_i$ are not all distinct if and only if either (i) $E=E_0$, (ii) $\Gamma$ is bipartite and $E=E_d$, or (iii) $\Gamma$ is antipodal and $E=E_i$ for an even $i$; cf.~\cite[pp.~140--143]{BCN1989B}.

\begin{assumption}\label{assumption}
For the rest of this paper, let $(X,\mathcal{R})$ be a symmetric association scheme with $d\geqslant 3$ classes, and assume that $(X,\mathcal{R})$ is $P$-polynomial with respect to the ordering $A_0,A_1,\dots,A_d$ of the adjacency matrices, unless otherwise stated.
Write $\Gamma=(X,R_1)$ and $A=A_1$, and recall the intersection array of $\Gamma$ from \eqref{intersection array}.
Let $\partial$ denote the path-length distance function on $X$ with respect to $\Gamma$.
Let $E$ be a primitive idempotent of $(X,\mathcal{R})$ with multiplicity $m\geqslant 3$, and assume that the dual eigenvalues $\theta_0^*,\theta_1^*,\dots,\theta_d^*$ (and hence the cosines $\sigma_0,\sigma_1,\dots,\sigma_d$) of $(X,\mathcal{R})$ associated with $E$ are distinct.
Let $\tilde{X}=\{\bm{\xi}_x:x\in X\}\subset S^{m-1}$ be the spherical embedding of $(X,\mathcal{R})$ with respect to $E$.
Note that the $\bm{\xi}_x$ are distinct.
\end{assumption}

With the above assumption, we have
\begin{equation}\label{sigma1 unequal to -1}
    \sigma_1\ne-1.
\end{equation}
Indeed, that $\sigma_1=-1$ means that $\bm{\xi}_x=-\bm{\xi}_y$ for adjacent vertices $x$ and $y$ in $\Gamma$, implying that $|X|=2$ and $d=1$, i.e., $\Gamma=K_2$, which is not the case here.
We also note that $K_2$ $(k=1)$ and the cycles $(k=2)$ satisfy $m_i\leqslant 2$ for all $i$.
Since $m\geqslant 3$, it follows that
\begin{equation}\label{k at least 3}
    k\geqslant 3.
\end{equation}

It follows from (AS\ref{AS2}) that $A_i\circ A_j=\delta_{i,j}A_i$ $(0\leqslant i,j\leqslant d)$, where $\circ$ denotes the entrywise (or \emph{Hadamard} or \emph{Schur}) product.
Hence, the adjacency algebra $\bm{A}$ is also closed under $\circ$, and there exist scalars $q_{i,j}^h\in \mathbb{R}$ $(0\leqslant i,j,h\leqslant d)$, known as the \emph{Krein parameters} of $(X,\mathcal{R})$, such that
\begin{equation*}
    E_i\circ E_j=\frac{1}{|X|}\sum_{h=0}^d q_{i,j}^h E_h \qquad (0\leqslant i,j\leqslant d).
\end{equation*}
The matrix $E_i\circ E_j$ is a principal submatrix of the positive semidefinite matrix $E_i\otimes E_j$ and thus is also positive semidefinite, from which it follows that the $q_{i,j}^h$ are nonnegative.
An ordering $E_0,E_1,\dots,E_d$ of the primitive idempotents is called $Q$-\emph{polynomial} if for all $i,j,h$ $(0\leqslant i,j,h\leqslant d)$, we have $q_{i,j}^h=0$ (resp.~$q_{i,j}^h>0$) whenever one of $i,j,h$ is greater than (resp.~equal to) the sum of the other two.
In this case, we have
\begin{equation}\label{3-term recurrence*}
 E\circ E_i=\frac{1}{|X|}\big(b_{i-1}^*E_{i-1}+a_i^*E_i+c_{i+1}^*E_{i+1}\big) \qquad (0\leqslant i\leqslant d),
\end{equation}
where we set $E=E_1$,
\begin{equation*}
    a_i^*=q_{1,i}^i, \qquad b_i^*=q_{1,i+1}^i, \qquad c_i^*=q_{1,i-1}^i
\end{equation*}
for $0\leqslant i\leqslant d$, and where $b_{-1}^*E_{-1}=c_{d+1}^*E_{d+1}:=0$.
It is not difficult to show that $m(=m_1)=b_0^*$, $c_1^*=1$, $c_0^*=b_d^*=0$, and $a_i^*=m-b_i^*-c_i^*$ $(0\leqslant i\leqslant d)$.
Recall the \emph{Krein array} of $(X,\mathcal{R})$ with respect to the above $Q$-polynomial ordering:
\begin{equation}\label{Krein array}
    \{b_0^*,b_1^*,\dots,b_{d-1}^*;c_1^*,c_2^*,\dots,c_d^*\}.
\end{equation}
By \eqref{3-term recurrence*}, the $Q$-polynomial ordering $E_0,E_1,\dots,E_d$ is determined by $E=E_1$, and hence we also say that the primitive idempotent $E$ is $Q$-\emph{polynomial}.
We note that the $d+1$ dual eigenvalues $\theta_0^*,\theta_1^*,\dots,\theta_d^*$ associated with $E$ are distinct because $E$ generates the adjacency algebra $\bm{A}$ with respect to $\circ$ by \eqref{3-term recurrence*}.
We say that $(X,\mathcal{R})$ is $Q$-\emph{polynomial} if it possesses a $Q$-polynomial idempotent.

For the rest of this section, assume that $E_0,E_1,\dots,E_d$ is a $Q$-polynomial ordering, and set $E=E_1$.
By computing $\operatorname{trace}(AE)$ in two ways using \eqref{eigenvalues}, \eqref{theta0}, \eqref{dual eigenvalues}, and \eqref{theta0*}, we find
\begin{equation}\label{Askey-Wilson}
    \frac{\theta_1}{k}=\frac{\theta_1^*}{m}.
\end{equation}
Moreover, by a theorem of Leonard \cite{Leonard1982SIAM}, \cite[Section~3.5]{BI1984B}, there exists a scalar $\beta\in\mathbb{R}$ such that the eigenvalues $\theta_0,\theta_1,\dots,\theta_d$ and the dual eigenvalues $\theta_0^*,\theta_1^*,\dots,\theta_d^*$ satisfy the following recurrence relations:
\begin{align}
    \theta_{i-2}-(\beta+1)\theta_{i-1}+(\beta+1)\theta_i-\theta_{i+1} &= 0 \qquad (2\leqslant i\leqslant d-1), \label{beta-recurrence} \\
    \theta^*_{i-2}-(\beta+1)\theta^*_{i-1}+(\beta+1)\theta^*_i-\theta^*_{i+1} &= 0 \qquad (2\leqslant i\leqslant d-1). \label{beta-recurrence*}
\end{align}
See also \cite[Section~8.1]{BCN1989B} and \cite[Theorem~1.9]{Terwilliger2001LAA}. 
From \eqref{beta-recurrence} and \eqref{beta-recurrence*} it follows that there exist scalars $\gamma,\gamma^*\in\mathbb{R}$ such that
\begin{align}
    \theta_{i-1}-\beta\theta_i+\theta_{i+1} &= \gamma\phantom{{}^*} \qquad (1\leqslant i\leqslant d-1), \label{beta,gamma-recurrence} \\
    \theta^*_{i-1}-\beta\theta^*_i+\theta^*_{i+1} &= \gamma^* \qquad (1\leqslant i\leqslant d-1). \label{beta,gamma-recurrence*}
\end{align}
Finally, from \eqref{beta,gamma-recurrence} and \eqref{beta,gamma-recurrence*} it follows that there exist scalars $\varrho,\varrho^*\in\mathbb{R}$ such that
\begin{align}
    \theta_{i-1}^2-\beta\theta_{i-1}\theta_i+\theta_i^2-\gamma(\theta_{i-1}+\theta_i) &= \varrho\phantom{{}^*} \qquad (1\leqslant i\leqslant d), \notag \\
    \theta_{i-1}^{*2}-\beta\theta_{i-1}^*\theta_i^*+\theta_i^{*2}-\gamma^*(\theta_{i-1}^*+\theta_i^*) &= \varrho^* \qquad (1\leqslant i\leqslant d). \label{beta,gamma,rho-recurrence*}
\end{align}
See \cite[Section~8]{Terwilliger2001LAA}.
We will use only the latter identity above.

\section{Vanishing polynomials}
\label{sec: vanishing polynomials}

For a nonempty finite subset $\Omega$ of $S^{m-1}$, define the (maximum) \emph{strength} $t(\Omega)$ of $\Omega$ to be the largest integer $t$ such that $\Omega$ is a spherical $t$-design in $S^{m-1}$.
Our polynomial method is based on the following observation due to Eiichi Bannai (personal communication):

\begin{lemma}\label{Bannai's observation}
Let $m$ be a positive integer and let $\Omega$ be a nonempty finite subset of $S^{m-1}$.
Let $f\in\mathbb{R}[\bm{\xi}]=\mathbb{R}[\xi_1,\xi_2,\dots,\xi_m]$ be a polynomial such that $f|_{S^{m-1}}$ is nonzero\footnote{Here, by `nonzero' we mean `not the zero function,' or `not identically zero.'} and that $f(\bm{\xi})=0$ for all $\bm{\xi}\in \Omega$.
Then, we have $t(\Omega)<2\deg f$.
\end{lemma}

\begin{proof}
Since $f|_{S^{m-1}}$ is nonzero, we have
\begin{equation*}
    \frac{1}{\omega_{m-1}} \int_{S^{m-1}} f^2(\bm{\xi})\,\mathrm{d}\sigma(\bm{\xi})>0,
\end{equation*}
whereas the average of the values $f^2(\bm{\xi})$ $(\bm{\xi}\in \Omega)$ equals zero.
Hence, $\Omega$ cannot be a spherical $(2\deg f)$-design.
\end{proof}

The above lemma enables us in the propositions below to bound $t(\tilde{X})$ for various classes of distance-regular graphs by constructing appropriate polynomials that vanish on $\tilde{X}$.
To show that these polynomials are nonzero on $S^{m-1}$, we will use (implicitly) the fact that the ring $\mathbb{R}[\bm{\xi}]$ is a unique factorization domain, and that the ideal consisting of the polynomials that vanish entirely on $S^{m-1}$ is generated by the irreducible quadratic polynomial $\bm{\xi}\cdot\bm{\xi}-1$, provided that $m>1$.
We begin with a warm-up:

\begin{proposition}\label{triangle}
Referring to Assumption~\ref{assumption}, assume that $a_1>0$.
Then, we have $t(\tilde{X})\leqslant 5$.
\end{proposition}

\begin{proof}
Let $x,x'$, and $x''$ be three mutually adjacent vertices in $\Gamma$, which is possible because $a_1>0$.
Consider the polynomial\footnote{This cubic polynomial was shown to the authors by Paul Terwilliger.}
\begin{equation*}
    f(\bm{\xi}) =(\bm{\xi}\cdot\bm{\xi}_x - \bm{\xi}\cdot\bm{\xi}_{x'})(\bm{\xi}\cdot\bm{\xi}_{x'} - \bm{\xi}\cdot\bm{\xi}_{x''})(\bm{\xi}\cdot\bm{\xi}_{x''} - \bm{\xi}\cdot\bm{\xi}_x).
\end{equation*}
Observe that each of the linear factors
\begin{equation*}
    \bm{\xi}\cdot\bm{\xi}_x - \bm{\xi}\cdot\bm{\xi}_{x'}, \quad \bm{\xi}\cdot\bm{\xi}_{x'} - \bm{\xi}\cdot\bm{\xi}_{x''}, \quad \bm{\xi}\cdot\bm{\xi}_{x''} - \bm{\xi}\cdot\bm{\xi}_x
\end{equation*}
is nonzero on $S^{m-1}$ because $\bm{\xi}_x,\bm{\xi}_{x'}$, and $\bm{\xi}_{x''}$ are distinct.
Pick any $z\in X$, and let $i=\partial(x,z)$, $j=\partial(x',z)$, and $h=\partial(x'',z)$.
Then, we have
\begin{equation*}
    f(\bm{\xi}_z)=(\sigma_i-\sigma_j)(\sigma_j-\sigma_h)(\sigma_h-\sigma_i).
\end{equation*}
Without loss of generality, assume that $i\leqslant j\leqslant h$.
Since $x'$ and $x''$ are neighbors of $x$ in $\Gamma$, we have $j,h\in\{i,i+1\}$, from which it follows that $f(\bm{\xi}_z)=0$.
By Lemma~\ref{Bannai's observation}, we have $t(\tilde{X})<2\cdot 3=6$.
\end{proof}

In the same vein, we can also prove the following:

\begin{proposition}\label{bipartite}
Referring to Assumption~\ref{assumption}, assume that $\Gamma$ is bipartite.
Then, we have $t(\tilde{X})\leqslant 5$.
\end{proposition}

\begin{proof}
Let $y\in X$, and let $x,x'$, and $x''$ be three distinct neighbors of $y$ in $\Gamma$, where we recall $k\geqslant 3$; cf.~\eqref{k at least 3}.
Observe that $x,x'$, and $x''$ are at distance two in $\Gamma$ from each other.
The rest of the proof proceeds exactly as in the proof of Proposition~\ref{triangle}, using the same degree-three polynomial $f(\bm{\xi})$.
The only difference is that we have $j,h\in\{i,i+2\}$ in the present case.
\end{proof}

The distance-regular graph $\Gamma$ is bipartite if and only if $a_0=a_1=\cdots=a_d=0$; see \cite[Proposition~4.2.2]{BCN1989B}.
We call $\Gamma$ \emph{almost bipartite} if $a_0=a_1=\cdots=a_{d-1}=0$ and $a_d>0$.
We note that if $\Gamma$ is almost bipartite, then the shortest odd cycles in $\Gamma$ have length $2d+1$.

\begin{proposition}
Referring to Assumption~\ref{assumption}, assume that $\Gamma$ is almost bipartite.
Then, we have $t(\tilde{X})\leqslant 5$.
\end{proposition}

\begin{proof}
Let $y,x,x',x''\in X$ be as in the proof of Proposition~\ref{bipartite}.
Again, $x,x'$, and $x''$ are at distance two in $\Gamma$ from each other because $a_1=0$.
Using the same notation and the same degree-three polynomial $f(\bm{\xi})$, we have $j,h\in\{i,i+1,i+2\}$ in this case.
Assume now that $f(\bm{\xi}_z)\ne 0$.
Then, we must have $j=i+1$ and $h=i+2$.
In particular, $i\leqslant d-2$.
By joining a shortest path of length $i$ connecting $z$ and $x$, the path $(x,y,x')$ of length two, and a shortest path of length $j=i+1$ connecting $x'$ and $z$, we obtain a closed walk of odd length $2i+3\leqslant 2d-1$, but this contradicts the above comment.
It follows that $f(\bm{\xi}_z)=0$ for all $z\in X$.
\end{proof}

We now establish the same upper bound $t(\tilde{X})\leqslant 5$ under the condition that the primitive idempotent $E$ is $Q$-polynomial, which improves on Suda's upper bound $t(\tilde{X})\leqslant 8$ \cite[Theorem~4.2]{Suda2011JCD}:

\begin{proposition}\label{Q-polynomial}
Referring to Assumption~\ref{assumption}, assume that $E$ is $Q$-polynomial.
Then, we have $t(\tilde{X})\leqslant 5$.
Moreover, if $\Gamma$ is bipartite, then $t(\tilde{X})\leqslant 3$.
\end{proposition}

\begin{proof}
Let $x$ and $y$ be adjacent vertices in $\Gamma$, and consider the polynomial
\begin{align*}
    f(\bm{\xi}) &=(\bm{\xi}\cdot\bm{\xi}_x - \bm{\xi}\cdot\bm{\xi}_y) \label{degree-four polynomial} \\
    &\phantom{{}={}}\times \left( (\bm{\xi}\cdot\bm{\xi}_x)^2-\beta(\bm{\xi}\cdot\bm{\xi}_x)(\bm{\xi}\cdot\bm{\xi}_y)+(\bm{\xi}\cdot\bm{\xi}_y)^2-\frac{\gamma^*}{m}(\bm{\xi}\cdot\bm{\xi}_x+\bm{\xi}\cdot\bm{\xi}_y)-\frac{\varrho^*}{m^2}\right),
\end{align*}
where the scalars $\gamma^*$ and $\varrho^*$ are from \eqref{beta,gamma-recurrence*} and \eqref{beta,gamma,rho-recurrence*}.
The linear factor $\bm{\xi}\cdot\bm{\xi}_x - \bm{\xi}\cdot\bm{\xi}_y$ is nonzero on $S^{m-1}$ because $\bm{\xi}_x$ and $\bm{\xi}_y$ are distinct.
Likewise, the quadratic factor is also nonzero on $S^{m-1}$.
To see this, let $\bm{\xi}_1$ and $\bm{\xi}_2$ be points on $S^{m-1}$ such that $\bm{\xi}_1\cdot\bm{\xi}_x=\bm{\xi}_1\cdot\bm{\xi}_y=\bm{\xi}_2\cdot\bm{\xi}_x=0$ and $c:=\bm{\xi}_2\cdot\bm{\xi}_y\ne 0$, where we recall that $m\geqslant 3$, and that $\bm{\xi}_y\ne -\bm{\xi}_x$, or equivalently, $\sigma_1\ne -1$; cf.~\eqref{sigma1 unequal to -1}.
Then, we have
\begin{align*}
    (\bm{\xi}_1\cdot\bm{\xi}_x)^2-\beta(\bm{\xi}_1\cdot\bm{\xi}_x)(\bm{\xi}_1\cdot\bm{\xi}_y)+(\bm{\xi}_1\cdot\bm{\xi}_y)^2 &= 0, \\
    (\bm{\xi}_2\cdot\bm{\xi}_x)^2-\beta(\bm{\xi}_2\cdot\bm{\xi}_x)(\bm{\xi}_2\cdot\bm{\xi}_y)+(\bm{\xi}_2\cdot\bm{\xi}_y)^2 &= c^2>0,
\end{align*}
and hence the degree-two homogeneous part of the quadratic factor is not a scalar multiple of $\bm{\xi}\cdot\bm{\xi}$.
In particular, this factor is not a scalar multiple of $\bm{\xi}\cdot\bm{\xi}-1$.

Pick any $z\in X$, and let $i=\partial(x,z)$, $j=\partial(y,z)$.
Then, we have
\begin{equation*}
    f(\bm{\xi}_z)=(\sigma_i-\sigma_j)\left(\sigma_i^2-\beta\sigma_i\sigma_j+\sigma_j^2-\frac{\gamma^*}{m}(\sigma_i+\sigma_j)-\frac{\varrho^*}{m^2}\right).
\end{equation*}
Without loss of generality, assume that $i\leqslant j$.
If $i=j$, then $\sigma_i-\sigma_j=0$.
On the other hand, if $i\ne j$, then we have $i=j-1$, and
\begin{align*}
    \MoveEqLeft[4] \sigma_i^2-\beta\sigma_i\sigma_j+\sigma_j^2-\frac{\gamma^*}{m}(\sigma_i+\sigma_j)-\frac{\varrho^*}{m^2} \\
    &= \frac{1}{m^2}\big(\theta_{j-1}^{*2}-\beta\theta_{j-1}^*\theta_j^*+\theta_j^{*2}-\gamma^*(\theta_{j-1}^*+\theta_j^*)-\varrho^*\big) \\
    &=0
\end{align*}
by \eqref{beta,gamma,rho-recurrence*}.
It follows that $f(\bm{\xi}_z)=0$ for all $z\in X$.
By Lemma~\ref{Bannai's observation}, $t(\tilde{X})<2\cdot 3=6$.

If $\Gamma$ is bipartite, then the quadratic factor alone already vanishes on all the $\bm{\xi}_z$ because we always have $i\ne j$ above.
It follows that $t(\tilde{X})<2\cdot 2=4$ in this case.
\end{proof}

The conditions on $\Gamma$ in the above propositions are determined by its intersection array.
We end this section with another construction of a polynomial that depends on the geometric structure of $\Gamma$.
Recall that the (column) \emph{characteristic vector} $\chi=\chi_Y$ of a subset $Y$ of $X$ is the column vector whose entries are indexed by $X$, such that the $x$-entry $\chi_x$ equals $1$ if $x\in Y$ and $0$ otherwise $(x\in X)$.

\begin{proposition}\label{dual width one}
Referring to Assumption~\ref{assumption}, assume that there is a nonempty proper subset $Y$ of $X$ whose characteristic vector $\chi=\chi_Y$ satisfies $\chi=E_0\chi+E\chi$.
Then, we have $t(\tilde{X})\leqslant 3$.
\end{proposition}

\begin{proof}
Note that $0<|Y|<|X|$.
By \eqref{E0}, we have $E_0\chi=|Y||X|^{-1}\bm{1}$, where $\bm{1}$ is the all ones vector.
This implies that
\begin{equation}\label{two values}
    \frac{|X|}{m}E\chi=\frac{|X|}{m}\chi-\frac{|Y|}{m}\bm{1} = \frac{|X\setminus Y|}{m}\chi-\frac{|Y|}{m}(\bm{1}-\chi).
\end{equation}
Let $\bm{\xi}_Y$ be the sum of the vectors $\bm{\xi}_x$ $(x\in Y)$, and consider the polynomial
\begin{equation*}
    f(\bm{\xi})=\left(\bm{\xi}\cdot\bm{\xi}_Y-\frac{|X\setminus Y|}{m}\right)\!\!\left(\bm{\xi}\cdot\bm{\xi}_Y+\frac{|Y|}{m}\right).
\end{equation*}
Then, it follows from \eqref{two values} that the two linear factors are nonzero on $S^{m-1}$ and that
\begin{equation*}
    f(\bm{\xi}_z)=\left(\frac{|X|}{m}(E\chi)_z-\frac{|X\setminus Y|}{m}\right)\!\!\left(\frac{|X|}{m}(E\chi)_z+\frac{|Y|}{m}\right)=0
\end{equation*}
for all $z\in X$.
By Lemma~\ref{Bannai's observation}, $t(\tilde{X})<2\cdot 2=4$.
\end{proof}

When $E$ is $Q$-polynomial, a nonempty subset $Y$ of the vertex set $X$ satisfying the condition in Proposition~\ref{dual width one} is said to have \emph{dual width} one in the sense of Brouwer, Godsil, Koolen, and Martin \cite{BGKM2003JCTA}.
If, moreover, $Y$ has \emph{width} (at most) $d-1$, i.e., $\partial(x,y)<d$ for all $x,y\in Y$, then $Y$ is called a \emph{descendent} of $\Gamma$ with dual width one (or width $d-1$); cf.~\cite{Tanaka2011EJC}.
Descendents were studied in detail in \cite{BGKM2003JCTA,Tanaka2006JCTA,Tanaka2011EJC}.
Many classical families of distance-regular graphs, such as the Hamming, Johnson, Grassmann, bilinear forms, and dual polar graphs, are known to have descendents with dual width one.
Hence, we have $t(\tilde{X})\leqslant 3$ for these graphs.
For example, every facet of the $d$-cube (or the $d$-dimensional hypercube) induces a $(d-1)$-cube and is a descendent with dual width one.
Note, however, that the definition of a descendent depends on $E$.
For example, the Hermitian dual polar graph $[^2\!A_{2d-1}(2)]$ has two $Q$-polynomial idempotents.
For one of them, this graph has descendents with dual width one and thus $t(\tilde{X})\leqslant 3$, whereas for the other, it has none and $t(\tilde{X})=5$ as claimed in Theorem~\ref{main theorem}.

\section{Proof of Theorem~\ref{main theorem}}
\label{sec: classification of schemes with t=5}

Recall Proposition~\ref{Q-polynomial}.
In this section, we complete the proof of Theorem~\ref{main theorem}.
Assumption~\ref{assumption} is in effect throughout the section (including in the theorems cited below).
In addition, we assume that the primitive idempotent $E$ is $Q$-polynomial, and that
\begin{equation}\label{t at least 5}
    t(\tilde{X})=5.
\end{equation}
Recall the corresponding Krein array of $(X,\mathcal{R})$ from \eqref{Krein array}. 
Our proof is based on the following result due to Suda \cite{Suda2011JCD}:

\begin{theorem}[{\cite[Theorem~3.1]{Suda2011JCD}}]\label{t in terms of Krein numbers} 
The spherical embedding $\tilde{X}$ of $(X,\mathcal{R})$ with respect to $E$ is a spherical $t$-design in $S^{m-1}$ if and only if the following hold:
\begin{equation*}
    a_i^*=0 \quad \left(0\leqslant i\leqslant \left\lfloor \frac{t-1}{2}\right\rfloor \right), \qquad c_i^*=\frac{mi}{m+2i-2} \quad \left(0\leqslant i\leqslant \left\lceil \frac{t-1}{2}\right\rceil \right).
\end{equation*}
\end{theorem}

\noindent
We note that \cite[Theorem~3.1]{Suda2011JCD} does not require that $(X,\mathcal{R})$ is $P$-polynomial.
By Theorem~\ref{t in terms of Krein numbers} and \eqref{t at least 5}, we have
\begin{equation}\label{5-design parameters}
    a_0^*=a_1^*=a_2^*=0, \qquad (c_0^*,c_1^*,c_2^*)=\left(0,1,\frac{2m}{m+2}\right).
\end{equation}
See also \cite[Lemma~2]{Munemasa2004EJC}.

Recall the scalar $\gamma^*$ from \eqref{beta,gamma-recurrence*}.
Terwilliger and Vidunas \cite[Theorem~5.3]{TV2004JAA} showed that there exist scalars $\omega,\eta\in\mathbb{R}$ such that
\begin{equation}\label{Terwilliger and Vidunas formula}
    a_i^*(\theta_i-\theta_{i-1})(\theta_i-\theta_{i+1})=\gamma^*\theta_i^2+\omega\theta_i+\eta \qquad (0\leqslant i\leqslant d),
\end{equation}
where $\theta_{-1}$ and $\theta_{d+1}$ are defined respectively by setting $i=0$ and $i=d$ in \eqref{beta,gamma-recurrence}.
It should be remarked that \cite[Theorem~5.3]{TV2004JAA} is a general result about linear algebraic objects called \emph{Leonard pairs} \cite{Terwilliger2001LAA}, and that \eqref{Terwilliger and Vidunas formula} is a special case when the Leonard pair is induced on the trivial (or primary) module for the \emph{Terwilliger algebra} \cite{Terwilliger1992JAC} of $(X,\mathcal{R})$; see, e.g., \cite{Cerzo2010LAA}.

By \eqref{5-design parameters}, the LHS in \eqref{Terwilliger and Vidunas formula} equals zero for $i\in\{0,1,2\}$.
Since $\theta_0,\theta_1$,  and $\theta_2$ are distinct, it follows that $\gamma^*=\omega=\eta=0$.
Then, again by \eqref{Terwilliger and Vidunas formula}, this in turn implies that
\begin{equation*}
    a_0^*=a_1^*=\cdots=a_{d-1}^*=0,
\end{equation*}
because the eigenvalues $\theta_0,\theta_1,\dots,\theta_d$ are distinct.
Here, we do not rule out the possibility to have $\theta_d=\theta_{d+1}$.
We say that the $Q$-polynomial idempotent $E$ is $Q$-\emph{bipartite} if $a_0^*=a_1^*=\cdots=a_d^*=0$ \cite[p.~241]{BCN1989B}, and that it is \emph{almost} $Q$-\emph{bipartite} if $a_0^*=a_1^*=\cdots=a_{d-1}^*=0$ and $a_d^*>0$.
We note that $E$ is $Q$-bipartite precisely when the distance-regular graph $\Gamma$ is an antipodal double-cover; cf.~\cite[Theorem~8.2.4]{BCN1989B}.
The bipartite $Q$-polynomial antipodal double-covers were classified by Curtin and Nomura; see \cite[Theorem~42]{Curtin1998DM} and \cite[Theorem~1.2]{Nomura1995JCTB}.
However, we have $t(\tilde{X})=3$ if $\Gamma$ is one of these graphs by Proposition~\ref{Q-polynomial} and since $a_1^*=0$, so they do not occur here. 
The classification of nonbipartite $Q$-polynomial antipodal double-covers was given by Dickie and Terwilliger \cite{DT1996EJC} as follows:

\renewcommand{\theenumi}{\roman{enumi}}
\begin{theorem}[{\cite[Theorem~1.1]{DT1996EJC}}]
Assume that $\Gamma$ is not bipartite and that $E$ is $Q$-bipartite.
Then, $\Gamma$ is one of the following:
\begin{enumerate}
\item\label{1st} the halved $2d$-cube;
\item\label{2nd} the Johnson graph $J(2d,d)$;
\item\label{3rd} a nonbipartite Taylor graph \textup{($d=3$)};
\item\label{4th} an antipodal double-cover with $d=4$ and
\begin{equation*}
    (c_1,c_2,c_3,c_4)=\big(1,\beta \zeta,(\beta^2-1)(2\zeta-\beta+1),\beta(2\zeta+2\beta\zeta-\beta^2)\big),
\end{equation*}
where $\beta\geqslant 3$ and $\zeta\geqslant 3\beta/4$ are integers, and $\zeta$ divides $\beta^2(\beta^2-1)/2$.
\end{enumerate}
\end{theorem}

\noindent
See also \cite[Section~5.5.8]{DKT2016EJC}.
We have $b_i=c_{d-i}$ $(0\leqslant i\leqslant d)$ for these graphs; cf.~\cite[Proposition~4.2.2]{BCN1989B}.
We also note that the integer $\beta$ in \eqref{4th} above corresponds to the one in \eqref{beta-recurrence} and \eqref{beta-recurrence*}.
Concerning the almost $Q$-bipartite $Q$-polynomial distance-regular graphs, we have the following result due to Dickie \cite{Dickie1995D}:

\begin{theorem}[{\cite[Theorem~3.1.4]{Dickie1995D}}]
Assume that $E$ is almost $Q$-bipartite and that $d\geqslant 4$.
Then, $\Gamma$ is one of the following:
\begin{enumerate}
\addtocounter{enumi}{4}
\item\label{5th} the halved $(2d+1)$-cube;
\item\label{6th} the folded $(2d+1)$-cube;
\item\label{7th} a Hermitian dual polar graph $[^2\!A_{2d-1}(r)]$, where $r$ is a prime power.
\end{enumerate}
\end{theorem}

We now discuss the above seven families of distance-regular graphs \eqref{1st}--\eqref{7th} separately and then handle the remaining case where $E$ is almost $Q$-bipartite and $d=3$.
Recall the fixed $Q$-polynomial ordering $E_0,E=E_1,E_2,\dots,E_d$.
For convenience, we also write the eigenvalues in decreasing order as follows (cf.~\eqref{eigenvalues}):
\begin{equation}\label{natural ordering}
    \theta_0>\theta_{i_1}>\theta_{i_2}>\cdots>\theta_{i_d}.
\end{equation}

\subsection*{The halved $2d$-cube \eqref{1st}}

The halved $2d$-cube has eigenvalues (cf.~\eqref{natural ordering})
\begin{equation*}
    \theta_{i_j}=\frac{1}{2}\big((2d-2j)^2-2d\big) \qquad (0\leqslant j\leqslant d),
\end{equation*}
and we have $E=E_{i_1}$; cf.~\cite[Section~9.2D]{BCN1989B}.
The Krein array agrees with the intersection array of the folded $2d$-cube.
We have $m=2d$ and $c_2^*=2$, so \eqref{5-design parameters} is not satisfied, and hence this graph does not occur as $\Gamma$.
We may also invoke Proposition~\ref{dual width one} to obtain the conclusion.
The vertex set $X$ consists of the even-weight vectors in the $2d$-dimensional vector space $\mathbb{F}_2^{2d}$ over the finite field $\mathbb{F}_2$.
For every $i\in\{1,2,\dots,2d\}$ and every $a\in\mathbb{F}_2$, the set
\begin{equation}\label{descendent of H(2d,2)/2}
    \{x=(x_1,x_2,\dots,x_{2d})\in X:x_i=a\}
\end{equation}
is a descendent with dual width one; cf.~\cite[Section~8]{Tanaka2011EJC}.

\subsection*{The halved $(2d+1)$-cube \eqref{5th} and the folded $(2d+1)$-cube \eqref{6th}}

The halved (resp.~folded) $(2d+1)$-cube is isomorphic to the distance-$2$ (resp.~distance-$d$) graph of the folded (resp.~halved) $(2d+1)$-cube, so that they induce isomorphic association schemes; cf.~\cite[Section~4.2D]{BCN1989B}.
The halved $(2d+1)$-cube has eigenvalues (cf.~\eqref{natural ordering})
\begin{equation*}
    \theta_{i_j}=\frac{1}{2}\big((2d+1-2j)^2-2d-1\big) \qquad (0\leqslant j\leqslant d),
\end{equation*}
and we have either $E=E_{i_1}$ or $E=E_{i_2}$; cf.~\cite[Section~9.2D]{BCN1989B}.
When $E=E_{i_1}$, the Krein array agrees with the intersection array of the folded $(2d+1)$-cube and thus $E$ is almost $Q$-bipartite, and when $E=E_{i_2}$, the Krein array agrees with the intersection array of the halved $(2d+1)$-cube and thus $a_1^*>0$.
Assume now that $E=E_{i_1}$.
Then, we have $m=2d+1$ and $c_2^*=2$, so \eqref{5-design parameters} is not satisfied, and hence this graph does not occur as $\Gamma$.
We note that the subset of $X$ defined similarly to \eqref{descendent of H(2d,2)/2} by replacing $2d$ by $2d+1$ is not a descendent but still has dual width one.

\subsection*{The Johnson graph $J(2d,d)$ \eqref{2nd}}

The Johnson graph $J(2d,d)$ has eigenvalues (cf.~\eqref{natural ordering})
\begin{equation*}
    \theta_{i_j}=(d-j)^2-j \qquad (0\leqslant j\leqslant d),
\end{equation*}
and we have $E=E_{i_1}$; cf.~\cite[Section~9.1]{BCN1989B}.
The vertex set $X$ consists of the $d$-subsets of a $2d$-set $\Phi$, and for every $p\in\Phi$, the set $\{x\in X:p\in x\}$ is a descendent with dual width one; cf.~\cite[Theorem~8]{BGKM2003JCTA}.
It follows from Proposition~\ref{dual width one} that this graph does not occur as $\Gamma$.
It is also possible to describe $m$ and $c_2^*$, but we omit the details.

\subsection*{An antipodal double-cover with $d=4$ \eqref{4th}}

This distance-regular graph has eigenvalues (cf.~\eqref{natural ordering})
\begin{equation*}
    (\theta_0,\theta_{i_1},\dots,\theta_{i_4})=\big(\beta^2(2\zeta-\beta)+2\beta \zeta,\beta(2\zeta-\beta)+2\zeta,2\zeta-\beta,-\beta,-\beta^2\big),
\end{equation*}
and $E=E_{i_1}$ is $Q$-bipartite.
It follows that
\begin{align*}
    m &= \frac{\beta\big(\beta-\beta^3+2(\beta^2+\beta-1)\zeta\big)}{2\zeta}, \\
    c_2^* &= \frac{(\beta^2-\beta+2\zeta)\big(\beta-\beta^3+2(\beta^2+\beta-1)\zeta\big)}{2\zeta\big(\beta-\beta^2+2(\beta+1)\zeta\big)}.
\end{align*}
These values (including the eigenvalues) can be computed manually as follows.\footnote{We initially derived these expressions directly from the intersection array using Wolfram Mathematica (\url{https://www.wolfram.com/mathematica}).}
We are in (c) in \cite[Theorem~3\,(v)]{Terwilliger1988GC}.
The intersection array and the Krein array belong to Case (I) described in \cite[p.~370]{Terwilliger1992JAC} (see also \cite[pp.~264--265]{BI1984B}).
Let $q\in\mathbb{R}$ be such that $\beta=q+q^{-1}$, where we recall that $\beta\geqslant 3$.
Then, as discussed in \cite[p.~92]{Terwilliger1988GC}, the parameters $s,s^*,r_1$, and $r_2$ appearing in the expressions in Case (I) satisfy
\begin{equation*}
    r_2=-r_1,\qquad s=r_1^2, \qquad s^*=-q^{-d-1}.
\end{equation*}
By setting $d=4$ and $\zeta=c_2/\beta$, and using $c_1=c_1^*=1$, we routinely recover the intersection array and obtain the above expressions for $m=b_0^*$, $c_2^*$, and the $\theta_i$.
See \cite{DT1996EJC} for the combinatorial meaning of the positive integer $\zeta$ (denoted `$\eta$' there).

Assume now that $c_2^*=2m/(m+2)$, i.e., \eqref{5-design parameters} holds.
Then, by solving this equation for $\zeta$, we obtain $\zeta\in\{\zeta_1,\zeta_2\}$, where
\begin{equation*}
    \zeta_1=-\frac{\beta^2(\beta-1)}{2(\beta-2)}, \qquad \zeta_2=\frac{\beta \left(\beta^2-1\right)}{2 \left(\beta^2+\beta-1\right)}. 
\end{equation*}
Since $\zeta>0$ and $\zeta_1<0$, we have $\zeta\ne \zeta_1$.
On the other hand, since $\beta$, $\beta^2-1$, and $\beta^2+\beta-1$ are mutually coprime, it follows that $\zeta_2$ is not an integer and hence $\zeta\ne \zeta_2$.
We also see that $\zeta_2<3\beta/4$.
It follows that \eqref{5-design parameters} is never satisfied and hence this graph does not occur as $\Gamma$.

For completeness, we show that this graph has no other $Q$-polynomial idempotent.
For $E=E_{i_1}$, the $Q$-polynomial ordering of the $E_i$ coincides with the natural ordering, i.e., $i_j=j$ $(0\leqslant j\leqslant 4)$.
By a result of Suzuki \cite[Theorem~1]{Suzuki1998JACb}, if there is another $Q$-polynomial ordering $E_0,E_{\nu_1},\dots,E_{\nu_4}$, then $(\nu_1,\nu_2,\nu_3,\nu_4)$ is one of the following:
\begin{equation*}
    (2,4,3,1), \quad (4,1,3,2), \quad (4,2,3,1), \quad (3,2,1,4).
\end{equation*}
Since $E=E_1$ is $Q$-bipartite, it follows from \eqref{3-term recurrence*} that $E_0,E_2$, and $E_4$ span a proper subalgebra of $\bm{A}$ with respect to $\circ$.
Since $E_{\nu_1}$ generates $\bm{A}$ with respect to $\circ$, it follows that $\nu_1\not\in\{2,4\}$, and hence we have $(\nu_1,\nu_2,\nu_3,\nu_4)=(3,2,1,4)$.
However, using $\beta\geqslant 3$ and  $\zeta\geqslant 3\beta/4$, we can easily verify that
\begin{equation*}
    \frac{\theta_0-\theta_{\nu_3}}{\theta_{\nu_1}-\theta_{\nu_2}} < \frac{\theta_{\nu_1}-\theta_{\nu_4}}{\theta_{\nu_2}-\theta_{\nu_3}},
\end{equation*}
from which it follows that the ordering $\theta_0,\theta_{\nu_1},\dots,\theta_{\nu_4}$ of the eigenvalues does not satisfy the recurrence \eqref{beta-recurrence} (for any choice of the value `$\beta$'); cf.~\cite[Lemma~8.3]{Terwilliger2001LAA}.
It follows that $E=E_{i_1}$ is a unique $Q$-polynomial idempotent.

\subsection*{A Hermitian dual polar graph $[^2\!A_{2d-1}(r)]$ \eqref{7th}}

Let $r$ be a prime power.
The Hermitian dual polar graph $[^2\!A_{2d-1}(r)]$ has eigenvalues (cf.~\eqref{natural ordering})
\begin{equation*}
    \theta_{i_j}=r\frac{r^{2(d-j)}-1}{r^2-1}-\frac{r^{2j}-1}{r^2-1} \qquad (0\leqslant j\leqslant d),
\end{equation*}
and we have either $E=E_{i_1}$ or $E=E_{i_d}$; cf.~\cite[Section~9.4]{BCN1989B}.
The vertex set $X$ consists of the maximal (i.e., $d$-dimensional) isotropic subspaces of the $2d$-dimensional vector spase $\mathbb{F}_{r^2}^{2d}$ over the finite field $\mathbb{F}_{r^2}$ with respect to a fixed nondegenerate Hermitian form.
Let $v$ be any one-dimensional isotropic subspace.
Then, the set $\{x\in X:v\subset x\}$ is a descendent with dual width one when $E=E_{i_1}$, and Proposition~\ref{dual width one} applies; cf.~\cite[Theorem~1]{Tanaka2006JCTA}.
Hence, we assume that $E=E_{i_d}$, which is almost $Q$-bipartite.
This case was worked out in detail by Munemasa \cite[p.~266]{Munemasa2004EJC}.
We have
\begin{equation*}
    m=r\frac{r^{2d-1}+1}{r+1}, \qquad c_2^*=\frac{(r-1)(r^{2d-1}+1)}{r(r^{2d-3}+1)},
\end{equation*}
and it is easy to see that \eqref{5-design parameters} is satisfied precisely when $r=2$.

\subsection*{A nonbipartite Taylor graph \eqref{3rd}}

Assume that $\Gamma$ is a Taylor graph, which is an antipodal double-cover with $d=3$.
The intersection array of $\Gamma$ is of the form $\{k,\mu,1;1,\mu,k\}$.
We also assume that $\mu<k-1$, i.e., $\Gamma$ is not bipartite.
We have (cf.~\eqref{natural ordering}) $\theta_0=k$, $\theta_{i_2}=-1$, and $\theta_{i_1}(>1)$ and $\theta_{i_3}(<-1)$ are the solutions of the quadratic equation
\begin{equation*}
    \theta^2-(k-2\mu-1)\theta-k=0.
\end{equation*}
See \cite[Proposition~8.2.3, p.~431]{BCN1989B}.
The graph $\Gamma$ has two $Q$-polynomial idempotents $E_{i_1}$ and $E_{i_3}$, both of which are $Q$-bipartite.

Using \eqref{3-term recurrence} and $A_3^2=I$, we easily see that the distance-$2$ graph $\Gamma_2=(X,R_2)$ is again a Taylor graph with intersection array $\{k,k-\mu-1,1;1,k-\mu-1,k\}$, and that $A_2E_{i_j}=-\theta_{i_j}E_{i_j}$ for $j\in\{1,3\}$.
Since $k>-\theta_{i_3}>-1>-\theta_{i_1}$, replacing $\Gamma$ by $\Gamma_2$ interchanges the roles of $E_{i_1}$ and $E_{i_3}$.
Hence, we may assume without loss of generality that $E=E_{i_1}$.
Then, the $Q$-polynomial ordering of the $E_i$ coincides with the natural ordering, i.e., $i_j=j$ $(0\leqslant j\leqslant 3)$, and we have
\begin{equation*}
    m_0=1, \qquad m=m_1=\frac{\theta_3(k+1)}{\theta_3-\theta_1}, \qquad m_2=k, \qquad m_3=\frac{\theta_1(k+1)}{\theta_1-\theta_3}.
\end{equation*}

Assume now that $c_2^*=2m/(m+2)$, i.e., \eqref{5-design parameters} holds.
By setting $i=1$ in \eqref{3-term recurrence*} and computing the trace of both sides using $b_0^*=m$ and $a_1^*=0$, we find
\begin{equation*}
    k=m_2=\frac{m(m-1)}{c_2^*}=\frac{(m-1)(m+2)}{2}.
\end{equation*}
Then, using the above expressions for $m$ and $k$, and also $k-2\mu-1=\theta_1+\theta_3$ and $\theta_1\theta_3=-k$, we routinely obtain
\begin{gather*}
    \theta_1 = \frac{(m-1)\sqrt{m+2}}{2}, \qquad \theta_3=-\sqrt{m+2}, \\
    \mu = \frac{m^2+m-4-(m-3)\sqrt{m+2}}{4}.
\end{gather*}
As mentioned in Introduction, $\tilde{X}$ is a tight spherical $5$-design (cf.~\eqref{Fisher bound}) in this case, because
\begin{equation*}
    |X|=2(k+1)=m(m+1)=2\binom{m+1}{2}.
\end{equation*}

We end the discussion on this family with some comments, which we will use in Section~\ref{sec: concluding remarks}.

\begin{note}\label{icosahedron is special}
Here, we claim that, given an association scheme as in Theorem~\ref{main theorem}\,\eqref{Taylor}, 
the two identities involving $k,\mu$, and $m$ uniquely determine the nonbipartite Taylor graph up to isomorphism and the tight spherical $5$-design up to orthogonal transformation.
Indeed, assume first that the distance-$2$ graph $\Gamma_2$ has the same intersection array as that of $\Gamma$, i.e., $\mu=k-\mu-1$.
This is equivalent to $\theta_1=-\theta_3$, which holds if and only if $m=3$, and so $k=5$ and $\mu=2$.
It is easy to show that $\Gamma$ and $\Gamma_2$ are both isomorphic to (the $1$-skeleton of) the regular icosahedron.
Next, assume that $m=m_3$.
Then, we again have $\theta_1=-\theta_3$, and thus we are in the same situation as above.
In particular, the spherical embeddings with respect to $E_{i_1}$ and $E_{i_3}$ both consist of the vertices of the regular icosahedron up to orthogonal transformation.
\end{note}

\subsection*{The remaining case with $d=3$}

From now on, we assume that $E$ is almost $Q$-bipartite and that $d=3$.
Moreover, we also assume that $\Gamma$ is neither the halved $7$-cube \eqref{5th} nor the folded $7$-cube \eqref{6th}, which we have already ruled out.
Then, by \cite[Lemma~3.2.1]{Dickie1995D}, there exist scalars $q,s^*\in\mathbb{C}$ with
\begin{gather*}
    |q|\geqslant 1, \\
    q^i\ne 1 \ \  (1\leqslant i\leqslant 6), \quad s^*q^i\ne 1 \ \ (2\leqslant i\leqslant 6), \quad s^*q^i\ne -1 \ \ (1\leqslant i\leqslant 7),
\end{gather*}
such that
\begin{equation*}
    m = -\frac{(q^6-q)(s^*q+1)}{(q-1)(s^*q^7+1)}, \qquad c_2^* = \frac{(q^6-q)(q^2-1)(s^*q^6+1)}{q^2(q-1)(q^3-1)(s^*q^7+1)}.
\end{equation*}
We note that we are in (d) in \cite[Theorem~3\,(v)]{Terwilliger1988GC}.
The scalar $q$ satisfies $q+q^{-1}=\beta$, where $\beta$ is from \eqref{beta-recurrence} and \eqref{beta-recurrence*}.
The intersection array and the Krein array belong to Case (I) described in \cite[p.~370]{Terwilliger1992JAC}, and, as discussed in \cite[p.~93]{Terwilliger1988GC}, the parameters $s,r_1$, and $r_2$ appearing in the expressions in Case (I) satisfy
\begin{equation*}
    s=q^{-2d-2}=q^{-8}, \qquad r_1=-s^*, \qquad r_2=-q^{-d-1}=-q^{-4}.
\end{equation*}

Assume now that $c_2^*=2m/(m+2)$, i.e., \eqref{5-design parameters} holds.
Using $q,q^5\ne 1$, we can rewrite this equation as follows:
\begin{equation}\label{quadratic equation}
    q^8(2q+1)s^{\ast 2}-q^2(q^5-1)s^*-q-2=0.
\end{equation}
Our aim is to show that
\begin{equation*}
    s^*=0, \qquad q=-2.
\end{equation*}
Then, we will see that the intersection array coincides with that of the Hermitian dual polar graph $[^2\!A_5(2)]$.
By \cite[Theorem~9.4.7]{BCN1989B}, we will then conclude that $\Gamma$ is isomorphic to $[^2\!A_5(2)]$.

First, we claim that $a_1\ne 0$ and $p_{2,3}^3\ne 0$, conditions assumed in many of the lemmas in \cite[Chapter~3]{Dickie1995D} that we will use.
Let us assume that $a_1=0$.
Then, by \cite[Lemma~3.3.1]{Dickie1995D}, we have $s^*=q^{-8}$, and it follows from \eqref{quadratic equation} that $(q+1)^2(q^7-1)=0$, so that $q^7=1$ because $q\ne -1$.
However, this would give us that $m=2$, which we have excluded in Assumption~\ref{assumption}.
We have now shown that $a_1\ne 0$.
Likewise, if $p_{2,3}^3=0$, then by \cite[Lemma~3.5.1]{Dickie1995D}, we have $s^{*2}=q^{-9}$, and it follows from \eqref{quadratic equation} and $q^3\ne 1$ that $q^7=1$ and $s^*=q^{-1}$, again implying that $m=2$, a contradiction.
The claim is proved.

By solving \eqref{quadratic equation} for $s^*$, we have
\begin{equation*}
    s^*=\frac{q^5-1\pm\sqrt{(q^5-1)^2+4q^4(2q+1)(q+2)}}{2q^6(2q+1)},
\end{equation*}
where we note that the denominator is nonzero because $|q|\geqslant 1$.
For the moment, we assume that \begin{equation*}
    s^*\ne 0.    
\end{equation*}
By \cite[Lemmas~3.3.5, 3.3.6, 3.5.2, 3.5.3]{Dickie1995D}, we have $q,s^*\in\mathbb{R}$ with $s^*<0$ and $q<-1$, from which it follows that
\begin{equation}\label{s*}
    s^*=\frac{q^5-1+\sqrt{(q^5-1)^2+4q^4(2q+1)(q+2)}}{2q^6(2q+1)}, \qquad q<-2,
\end{equation}
because $q^5-1<0$ and $2q+1<0$.
Observe that (cf.~\cite[Lemma~3.5.4]{Dickie1995D})
\begin{equation}\label{s^*q^6}
    s^*q^6=\frac{-2q^4(q+2)}{q^5-1-\sqrt{(q^5-1)^2+4q^4(2q+1)(q+2)}}\rightarrow -1 \qquad (q\rightarrow -\infty).
\end{equation}

We follow the notation of Dickie \cite[Lemma~3.3.3]{Dickie1995D}.
For $i\in\{2,3\}$, we define the nonnegative integer $L_i\in\mathbb{Z}$ by
\begin{equation*}
    L_i=\big|\big\{w\in X:\partial(x,w)=i-1,\,\partial(y,w)=\partial(z,w)=1\big\}\big|,
\end{equation*}
where $\partial(x,y)=i$, $\partial(x,z)=i-1$, and $\partial(y,z)=2$.
(Here, such a triple $x,y,z$ exists by \cite[Lemma~3.3.2]{Dickie1995D}.)
By \cite[(3.40)]{Dickie1995D},
\begin{align}
    L_2 &= \frac{(s^*q^2-1)(s^*q^3-1)(s^*q^6-1)({s^*}^2q^{12}-1)}{
	(s^*q^4-1)(s^*q^5-1)(s^*q^4-1)({s^*}^2q^{10}-1)}, \label{L2} \\
    L_3 &= \frac{(s^*q^2-1)(s^*q^3-1)(s^*q^8-1)({s^*}^2q^{12}-1)}{
(s^*q^4-1)(s^*q^5-1)(s^*q^6-1)({s^*}^2q^{10}-1)}. \label{L3} \end{align}
By \cite[Lemma~3.3.4]{Dickie1995D}, we have
\begin{equation}\label{L2, L3 positive}
    L_2>0, \qquad L_3>0.
\end{equation}
Let $g_2(q)$ (resp.~$g_3(q)$) denote the RHS in \eqref{L2} (resp.~\eqref{L3}), viewed as a function of $q\in(-\infty,-2]$, where $s^*=s^*(q)$ is as in \eqref{s*}.
We note that
\begin{equation}\label{at q=-2}
    s^*(-2)=0, \qquad g_2(-2)=g_3(-2)=1.
\end{equation}
It follows from \eqref{s^*q^6} that
\begin{equation*}
    g_2(q)\rightarrow 0 \qquad (q\rightarrow -\infty).
\end{equation*}
Hence, $\Gamma$ cannot exist when $q$ is sufficiently small.

Figures \ref{graph of g2} and \ref{graph of g3} show the graphs of the functions $g_2(q)$ and $g_3(q)$, respectively, produced by SageMath.\footnote{\url{https://www.sagemath.org/}}
Figure \ref{graph of g2} suggests that there exists only one value of $q<-2$ such that $g_2(q)$ is an integer: we have $g_2(q)=1$ at $q\approx -6.15$.
However, we then see from Figure \ref{graph of g3} that $9<g_3(q)<10$ for this particular value of $q$, and hence $g_2(q)$ and $g_3(q)$ cannot both be integers for $q<-2$, from which we conclude that we must have $s^*=0$.
Then, it follows that the constant term of \eqref{quadratic equation} must be zero, i.e., $q=-2$.

\begin{figure}
\captionsetup{width=\linewidth}
\begin{minipage}{5.7cm}
\includegraphics[bb=0 0 453 338,width=5.7cm]{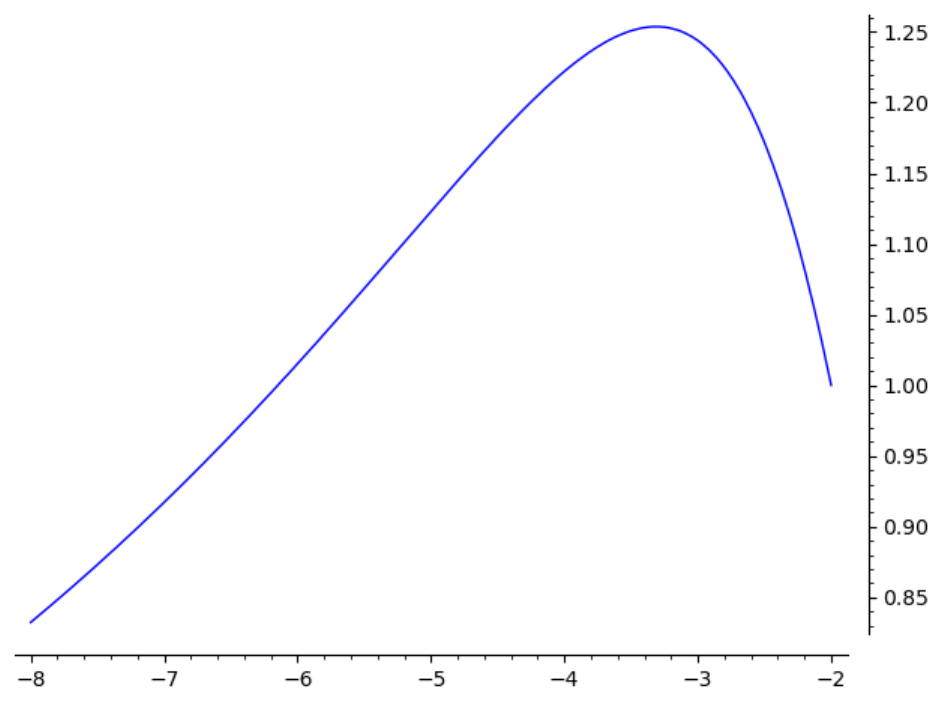}
\caption{Graph of $g_2(q)$}
\label{graph of g2}
\end{minipage} \qquad
\begin{minipage}{5.7cm}
\includegraphics[bb=0 0 453 338,width=5.7cm]{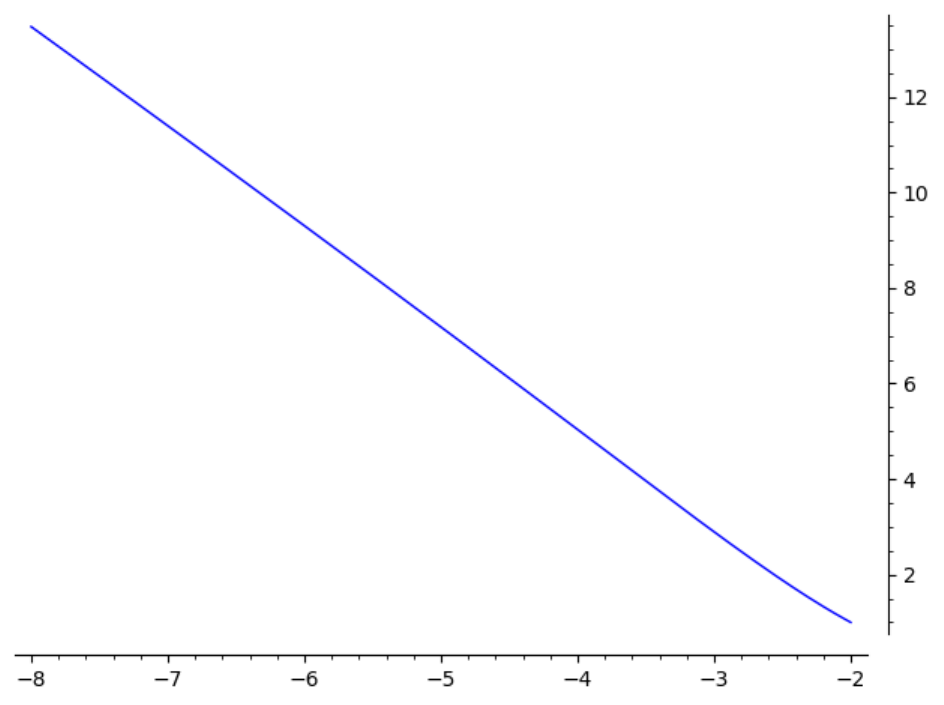}
\caption{Graph of $g_3(q)$}
\label{graph of g3}
\end{minipage}
\end{figure}

We now validate the above observations without using floating-point arithmetic.
It will be convenient to work with the following function instead of $g_3(q)$:
\begin{equation}\label{g32}
    g_{3/2}(q)=\frac{g_3(q)}{g_2(q)}=\frac{(s^*q^4-1)(s^*q^8-1)}{(s^*q^6-1)^2} \qquad (q\in(-\infty,-2]).
\end{equation}

\begin{lemma}\label{continuous}
The functions $g_2(q)$ and $g_{3/2}(q)$ are continuous on $(-\infty,2]$.
\end{lemma}

\begin{proof}
Since $q\leqslant -2$ and $s^*=s^*(q)\leqslant 0$, we have $s^*q^4-1<0$ and $s^*q^6-1<0$.
Moreover, we have $s^*q^5\geqslant 0$ and also (cf.~\eqref{s^*q^6})
\begin{equation*}
    s^*q^5 = \frac{2q^3(q+2)}{\sqrt{(q^5-1)^2+4q^4(2q+1)(q+2)}+1-q^5} \leqslant \frac{q^3(q+2)}{1-q^5} <-\frac{1}{q}\leqslant \frac{1}{2}.
\end{equation*}
It follows that the denominators of $g_2(q)$ and $g_{3/2}(q)$ never vanish, and hence the two functions are continuous.
\end{proof}

In the proofs of the following lemmas, we let $\psi(q)$ denote the LHS of \eqref{quadratic equation}.

\begin{lemma}\label{less than 2}
We have $g_2(q)<2$ for all $q\in(-\infty,-2]$.
\end{lemma}

\begin{proof}
Recall \eqref{at q=-2}.
Suppose, on the contrary, that $g_2(q)\geqslant 2$ for some $q\in(-\infty,-2)$.
Then, by Lemma~\ref{continuous} and the Intermediate Value Theorem, we have $g_2(\dot{q})=2$ for some $\dot{q}\in(-\infty,-2)$.
This implies that $\varphi_2(\dot{q})=0$, where (cf.~\eqref{L2})
\begin{align*}
    \varphi_2(q) &= (s^*q^2-1)(s^*q^3-1)(s^*q^6-1)(s^{*2}q^{12}-1) \\
    &\phantom{{}={}} -2(s^*q^4-1)(s^*q^5-1)(s^*q^4-1)({s^*}^2q^{10}-1).
\end{align*}

For the moment, let us view $q$ and $s^*$ as indeterminates, so $\varphi_2(q)=\varphi_2(q,s^*)$ and $\psi(q)=\psi(q,s^*)$ belong to the polynomial ring $\mathbb{Q}[q,s^*]$.
Let $\tilde{h}_2(q)$ be a generator of the principal ideal $\mathcal{I}\cap\mathbb{Q}[q]$ of $\mathbb{Q}[q]$, where $\mathcal{I}$ denotes the ideal of $\mathbb{Q}[q,s^*]$ generated by $\varphi_2(q,s^*)$ and $\psi(q,s^*)$.
We have $\tilde{h}_2(q)=(q-1)^4(q+1)^4h_2(q)$ up to a nonzero scalar in $\mathbb{Q}$, where
\begin{align*}
    h_2(q) &= 3q^{14}+11q^{13}+21q^{12}-11q^{11}-9q^{10}-86q^9+68q^8-36q^7 \\
    &\phantom{{}={}} +68q^6-86q^5-9q^4-11q^3+21q^2+11q+3.
\end{align*}
This is found by computing the Gr\"{o}bner basis of $\mathcal{I}$ for the lexicographic order with $s^*>q$; see, e.g., \cite[Chapter~3]{CLS2025B}.
We simply call this procedure of obtaining $\tilde{h}_2(q)$ \emph{eliminating} $s^*$ from $\varphi_2(q)$ and $\psi(q)$, which we will use three times below.
A SageMath code is given in Appendix \ref{sec: SageMath codes}.

Observe that $\tilde{h}_2(\dot{q})=0$.
Since $q^2-1>0$ for $q<-2$, it follows that $h_2(\dot{q})=0$.
However, since
\begin{align*}
    h_2(q) &= 3(q+2)^2q^{12}+9(q^2-1)q^{10}+68q^8+59q^6+9q^4(q^2-1)+21q^2+3 \\
    &\phantom{{}={}} -q\left(q^{12}+11q^{10}+86q^4(q^4+1)+36q^6+11(q^2-1)\right)
\\&>0
\end{align*}
for $q<-2$, this is impossible.
It follows that $g_2(q)<2$ for all $q<-2$.
\end{proof}

For the rest of this section, let
\begin{equation*}
    \delta_1=-6.1, \qquad \delta_2=-6.2.
\end{equation*}

\begin{lemma}\label{L2 unequal to 1}
We have $g_2(q)\ne 1$ for all $q\in(-\infty,\delta_2]\cup[\delta_1,-2)$.
\end{lemma}

\begin{proof}
Suppose that $g_2(\dot{q})=1$ for some $\dot{q}\in(-\infty,-2)$.
This implies that $\varphi_1(\dot{q})=0$, where (cf.~\eqref{L2})
\begin{align*}
    \varphi_1(q) &= (s^*q^2-1)(s^*q^3-1)(s^*q^6-1)(s^{*2}q^{12}-1) \\
    &\phantom{{}={}} -(s^*q^4-1)(s^*q^5-1)(s^*q^4-1)({s^*}^2q^{10}-1).
\end{align*}
Eliminating $s^*$ from $\varphi_1(q)$ and $\psi(q)$ gives $\tilde{h}_1(\dot{q})=0$, where
$\tilde{h}_1(q)=(q-1)^4(q+1)^4(q+2)h_1(q)$, and
\begin{align*}
    h_1(q) &= q^{12}+5q^{11}-5q^{10}+7q^9-30q^8+23q^7-28q^6+23q^5 \\
    &\phantom{{}={}} -30q^4+7q^3-5q^2+5q+1.
\end{align*}
We note that $h_1(\dot{q})=0$.

We now invoke Sturm's Theorem; see, e.g., \cite[Theorem~10.5.3]{RS2002B}.
Let $p_0(q)=h_1(q)$ and $p_1(q)=p_0'(q)$ (the derivative of $p_0(q)$).
For $2\leqslant i\leqslant 12$, let $p_i(q)\in\mathbb{Q}[q]$ be the remainder after dividing $-p_{i-2}(q)$ by $p_{i-1}(q)$.
For a fixed value of $q$, let $V(q)$ be the number of sign variations (after excluding $0$'s) in the sequence $(p_0(q),p_1(q),\dots,p_{12}(q))$.
We have $p_{12}(q)\approx 16.052\ne 0$, which implies that $h_1(q)$ is a square-free polynomial.
For any finite open interval $(\delta,\epsilon)$ such that $h_1(\delta),h_1(\epsilon)\ne 0$, Sturm's Theorem states that the number of distinct zeros in it equals $V(\delta)-V(\epsilon)$.
For $0\leqslant i\leqslant 12$, let $p_i(-\infty)$ denote the leading coefficient of $(-1)^{\deg p_i(q)}p_i(q)$, and define $V(-\infty)$ in the same manner.
Then, the formula still holds when $\delta=-\infty$, because $V(q)=V(-\infty)$ for $q\ll 0$.
Table \ref{sign changes for h1} gives the values of $p_i(q)$ $(0\leqslant i\leqslant 12)$ and $V(q)$ for $q\in\{-\infty,\delta_2,\delta_1,-2\}$.
It should be remarked that the computations of the values of the $p_i(q)$ in the table were performed all in $\mathbb{Q}$ using SageMath, but we rounded the results for readability.
What we need here is to determine the signs of the $p_i(q)$ correctly to find $V(q)$.
Observe also that $h_1(q)=p_0(q)\ne 0$ for $q\in\{\delta_2,\delta_1,-2\}$.
It follows that $\dot{q}$ is a unique zero of the polynomial $h_1(q)$ in the interval
$(-\infty,-2)$, and it lies in $(\delta_2,\delta_1)$.
This establishes the result.
\renewcommand{\arraystretch}{1.1}
\begin{table}
\centering
\small
\begin{tabular}{c|rrrr}
\hline\hline
$q$ & \multicolumn{1}{c}{$-\infty$} & \multicolumn{1}{c}{$\delta_2$} & \multicolumn{1}{c}{$\delta_1$} & \multicolumn{1}{c}{$-2$} \\
\hline
$p_0(q)$ & $1\phantom{.000}$ & $34580357.527$ & $-26298056.749$ & $-28565\phantom{.000}$ \\
$p_1(q)$ & $-12\phantom{.000}$ & $-718430082.853$ & $-507474344.973$ & $122781\phantom{.000}$ \\
$p_2(q)$ & $2.743$ & $311663029.626$ & $266643545.132$ & $12364.729$ \\
$p_3(q)$ & $-14.036$ & $246983445.454$ & $220775924.006$ & $85700.578$ \\
$p_4(q)$ & $-486.554$ & $-1259452941.839$ & $-1109548401.763$ & $-270893.523$ \\
$p_5(q)$ & $2.531$ & $925925.414$ & $827178.367$ & $443.253$ \\
$p_6(q)$ & $327.260$ & $20129158.376$ & $18298188.408$ & $35541.315$ \\
$p_7(q)$ & $3.831$ & $32296.305$ & $29732.648$ & $94.326$ \\
$p_8(q)$ & $-873.420$ & $-1276456.983$ & $-1195989.727$ & $-14321.836$ \\
$p_9(q)$ & $0.952$ & $214.397$ & $204.010$ & $6.607$ \\
$p_{10}(q)$ & $96.827$ & $4548.563$ & $4414.347$ & $578.867$ \\
$p_{11}(q)$ & $-4.078$ & $-23.327$ & $-22.919$ & $-6.200$ \\
$p_{12}(q)$ & $16.052$ & $16.052$ & $16.052$ & $16.052$ \\
\hline
$V(q)$ & \multicolumn{1}{c}{$8$} & \multicolumn{1}{c}{$8$} & \multicolumn{1}{c}{$7$} & \multicolumn{1}{c}{$7$} \\
\hline\hline
\end{tabular}
\caption{Sign variations of the Sturm sequence for $h_1(q)$}\label{sign changes for h1}
\end{table}
\end{proof}

\begin{lemma}\label{less than 10}
We have $g_{3/2}(q)<10$ for all $q\in(\delta_2,-2]$.
\end{lemma}

\begin{proof}
By \eqref{s^*q^6}, we have
\begin{equation*}
    g_{3/2}(q) \rightarrow \infty \qquad (q\rightarrow-\infty).
\end{equation*}
Since $g_{3/2}(-2)=g_3(-2)/g_2(-2)=1$ (cf.~\eqref{at q=-2}),
it follows from Lemma~\ref{continuous} and the Intermediate Value Theorem that $g_{3/2}(\dot{q})=10$ for some 
$\dot{q}\in(-\infty,-2)$.
This implies that $\varphi_{10}(\dot{q})=0$, where (cf.~\eqref{g32})
\begin{equation*}
    \varphi_{10}(q)=(s^*q^4-1)(s^*q^8-1)-10(s^*q^6-1)^2.
\end{equation*}
Eliminating $s^*$ from $\varphi_{10}(q)$ and $\psi(q)$ gives $\tilde{h}_{10}(\dot{q})=0$, where $\tilde{h}_{10}(q)=(q-1)^2(q+1)^3h_{10}(q)$, and
\begin{equation*}
    h_{10}(q)=9q^8-331q^6+19q^5-771q^4+19q^3-331q^2+9.
\end{equation*}
Note that $h_{10}(\dot{q})=0$.
By Sturm's Theorem and Table \ref{sign changes for h10}, which gives the (rounded) values of the corresponding Sturm sequence $(p_0(q),p_1(q),\dots,p_8(q))$ and the number of sign variations $V(q)$ for $q\in\{-\infty,\delta_2,-2\}$, we find that $\dot{q}$ is a unique zero of $h_{10}(q)$ in the interval $(-\infty,-2)$, and it lies in $(-\infty,\delta_2)$.
It follows that $g_{3/2}(q)<10$ for all $q\in(\delta_2,-2]$, as desired.
\renewcommand{\arraystretch}{1.1}
\begin{table}
\centering
\small
\begin{tabular}{c|rrr}
\hline\hline
$q$ & \multicolumn{1}{c}{$-\infty$} & \multicolumn{1}{c}{$\delta_2$} & \multicolumn{1}{c}{$-2$} \\
\hline
$p_0(q)$ & $9\phantom{.000}$ & $-480832.091$ & $-33291\phantom{.000}$ \\
$p_1(q)$ & $-72\phantom{.000}$ & $-6279552.897$ & $82080\phantom{.000}$ \\
$p_2(q)$ & $82.750$ & $5347485.587$ & $12771\phantom{.000}$ \\
$p_3(q)$ & $-2320.886$ & $-22167170.805$ & $-103347.088$ \\
$p_4(q)$ & $-267.727$ & $-405242.422$ & $-5205.535$ \\
$p_5(q)$ & $1349.263$ & $326729.176$ & $12325.324$ \\
$p_6(q)$ & $79.517$ & $3033.771$ & $304.595$ \\
$p_7(q)$ & $-886.668$ & $-5505.465$ & $-1781.459$ \\
$p_8(q)$ & $8.976$ & $8.976$ & $8.976$ \\
\hline
$V(q)$ & \multicolumn{1}{c}{$6$} & \multicolumn{1}{c}{$5$} & \multicolumn{1}{c}{$5$} \\
\hline\hline
\end{tabular}
\caption{Sign variations of the Sturm sequence for $h_{10}(q)$}\label{sign changes for h10}
\end{table}
\end{proof}

\begin{lemma}\label{greater than 9}
We have $g_{3/2}(q)>9$ for all $q\in(-\infty,\delta_1)$.
\end{lemma}

\begin{proof}
We have $g_{3/2}(\dot{q})=9$ for some 
$\dot{q}\in(-\infty,-2)$, so that $\varphi_9(\dot{q})=0$, where
\begin{equation*}
    \varphi_9(q)=(s^*q^4-1)(s^*q^8-1)-9(s^*q^6-1)^2.
\end{equation*}
Eliminating $s^*$ from $\varphi_9(q)$ and $\psi(q)$ gives $\tilde{h}_9(\dot{q})=0$, where $\tilde{h}_9(q)=(q-1)^2(q+1)^3h_9(q)$, and
\begin{equation*}
    h_9(q)=8q^8-262q^6+17q^5-622q^4+17q^3-262q^2+8.
\end{equation*}
Note that $h_9(\dot{q})=0$.
Table \ref{sign changes for h9} gives the values of the corresponding Sturm sequence and the number of sign variations for $q\in\{-\infty,\delta_1,-2\}$.
It follows that $\dot{q}$ is a unique zero of $h_9(q)$ in the interval $(-\infty,-2)$, and it lies in $(\delta_1,-2)$.
In particular, we have $g_{3/2}(q)>9$ for all $q\in(-\infty,\delta_1)$.
\renewcommand{\arraystretch}{1.1}
\begin{table}
\centering
\small
\begin{tabular}{c|rrr}
\hline\hline
$q$ & \multicolumn{1}{c}{$-\infty$} & \multicolumn{1}{c}{$\delta_1$} & \multicolumn{1}{c}{$-2$} \\
\hline
$p_0(q)$ & $8\phantom{.000}$ & $819854.579$ & $-26392\phantom{.000}$ \\
$p_1(q)$ & $-64\phantom{.000}$ & $-6148987.812$ & $64628\phantom{.000}$ \\
$p_2(q)$ & $65.500$ & $3868748.628$ & $10235\phantom{.000}$ \\
$p_3(q)$ & $-1875.272$ & $-16542019.816$ & $-83655.879$ \\
$p_4(q)$ & $-217.284$ & $-308305.523$ & $-4209.236$ \\
$p_5(q)$ & $1138.730$ & $262572.427$ & $10339.590$ \\
$p_6(q)$ & $66.682$ & $2460.885$ & $254.673$ \\
$p_7(q)$ & $-723.578$ & $-4421.706$ & $-1455.034$ \\
$p_8(q)$ & $7.973$ & $7.973$ & $7.973$ \\
\hline
$V(q)$ & \multicolumn{1}{c}{$6$} & \multicolumn{1}{c}{$6$} & \multicolumn{1}{c}{$5$} \\
\hline\hline
\end{tabular}
\caption{Sign variations of the Sturm sequence for $h_9(q)$}\label{sign changes for h9}
\end{table}
\end{proof}

We now return to the discussion of $\Gamma$.
By \eqref{s*}, \eqref{L2, L3 positive}, and Lemma~\ref{less than 2}, we have $L_2=g_2(q)=1$.
By Lemma~\ref{L2 unequal to 1}, we also find $q\in(\delta_2,\delta_1)$.
But then, it follows from Lemmas~\ref{less than 10} and \ref{greater than 9} that
\begin{equation*}
    L_3=g_3(q)=g_2(q)g_{3/2}(q)=g_{3/2}(q)\in (9,10),
\end{equation*}
so $L_3$ cannot be integral, which is a contradiction.
Hence, we conclude that we must have $s^*=0$ as expected.
This also forces the constant term of \eqref{quadratic equation} to be zero, i.e., $q=-2$.
The intersection array of $\Gamma$ is described in \cite[Lemma~3.2.1]{Dickie1995D}, or in (d) in \cite[Theorem~3\,(v)]{Terwilliger1988GC}.
Setting $s^*=0$ and $q=-2$, we obtain
\begin{equation*}
    b_i=\frac{2(64-4^i)}{3}, \qquad c_i=\frac{4^i-1}{3} \qquad (0\leqslant i\leqslant 3).
\end{equation*}
This is the intersection array of the Hermitian dual polar graph $[^2\!A_5(2)]$; cf.~\cite[Theorem~9.4.3]{BCN1989B}.
By \cite[Theorem~9.4.7]{BCN1989B}, $\Gamma$ is isomorphic to $[^2\!A_5(2)]$.

This completes the proof of Theorem~\ref{main theorem}.

\section{Concluding remarks}
\label{sec: concluding remarks}

In Section~\ref{sec: classification of schemes with t=5}, we used the identity \eqref{Terwilliger and Vidunas formula} due to Terwilliger and Vidunas \cite{TV2004JAA} to show that the $Q$-polynomial idempotent $E$ in question is $Q$-bipartite or almost $Q$-bipartite.
Here, we mention that the same argument proves the following proposition, which is dual to Miklavi\v{c}'s result \cite[Theorem~6.3]{Miklavic2004EJC}:

\begin{proposition}
Referring to Assumption~\ref{assumption}, assume that $E$ is $Q$-polynomial, and that $a_1^*=0$.
Then, one of the following holds:
\begin{enumerate}
\item $E$ is $Q$-bipartite;
\item $E$ is almost $Q$-bipartite;
\item $a_i^*>0$ for $2\leqslant i\leqslant d$.
\end{enumerate}
\end{proposition}

\begin{proof}
Suppose that (iii) does not hold, i.e., $a_j^*=0$ for some $j\in\{2,3,\dots,d\}$.
Then, the LHS in \eqref{Terwilliger and Vidunas formula} equals zero for $i\in\{0,1,j\}$.
Since $\theta_0,\theta_1$, and $\theta_j$ are distinct, it follows that $\gamma^*=\omega=\eta=0$.
Then, again by \eqref{Terwilliger and Vidunas formula}, this in turn implies that
\begin{equation*}
    a_0^*=a_1^*=\cdots=a_{d-1}^*=0,
\end{equation*}
because the eigenvalues $\theta_0,\theta_1,\dots,\theta_d$ are distinct.
Hence, we have (i) if $a_d^*=0$ and (ii) if $a_d^*>0$, as desired.
\end{proof}

\noindent
On the other hand, Dickie \cite[Theorem~4.1.1]{Dickie1995D} showed that for a $Q$-polynomial idempotent $E$, that $a_1^*>0$ implies that $a_i^*>0$ for $2\leqslant i\leqslant d-1$ (where we do not need the $P$-polynomial property).
The result dual to this, concerning the $a_i$ for a $P$-polynomial adjacency matrix, follows easily from counting arguments; cf.~\cite[Proposition~5.5.1]{BCN1989B}.

For the rest of this section, the following assumption is in effect:

\begin{assumption}\label{assumption: Q only}
Let $(X,\mathcal{R})$ be a $Q$-polynomial association scheme with $d\geqslant 1$ classes, and let $E$ be a corresponding $Q$-polynomial idempotent, with multiplicity $m\geqslant 1$.
Let $E_0,E_1,\dots,E_d$ be the $Q$-polynomial ordering such that $E=E_1$.
Recall the dual eigenvalues $\theta_0^*,\theta_1^*,\dots,\theta_d^*$ and the cosines $\sigma_0,\sigma_1,\dots,\sigma_d$ of $(X,\mathcal{R})$ associated with $E$.
Let $\tilde{X}=\{\bm{\xi}_x:x\in X\}\subset S^{m-1}$ be the spherical embedding of $(X,\mathcal{R})$ with respect to $E$.
\end{assumption}

Delsarte, Goethals, and Seidel showed that every tight spherical $t$-design $\Omega$ is the spherical embedding of a $Q$-polynomial association scheme with $\lceil t/2\rceil$-classes with respect to a $Q$-polynomial idempotent \cite[p.~380]{DGS1977GD} (see also \cite[p.~322]{Godsil1993B}), and that $\Omega$ is antipodal (i.e., $\Omega=-\Omega$) if $t$ is odd \cite[Theorem~6.8]{DGS1977GD}.
In Introduction, we also mentioned that the tight spherical $5$-designs are, up to orthogonal transformation, in bijection with the isomorphism classes of nonbipartite Taylor graphs as in Theorem~\ref{main theorem}\,\eqref{Taylor}.
This last fact is known to some experts, but we know of no literature that explicitly mentions it.
Here, we prove it as a corollary to slightly more general results.

\begin{proposition}\label{antipodal iff Q-bipartite}
Referring to Assumption~\ref{assumption: Q only}, the following are equivalent:
\begin{enumerate}
    \item\label{1st: Q-bipartite} $E$ is $Q$-bipartite;    \item\label{2nd: antipodal} $\tilde{X}$ is antipodal, i.e., $\tilde{X}=-\tilde{X}$.
\end{enumerate}
\end{proposition}

\begin{proof}
The implication \eqref{1st: Q-bipartite} $\Rightarrow$ \eqref{2nd: antipodal} was shown in (the proof of) \cite[Theorem~4.1]{MMW2007JAC}.

Assume that \eqref{2nd: antipodal} holds.
Arrange the adjacency matrices $A_0,A_1,\dots,A_d$ so that the cosines are in decreasing order $\sigma_0>\sigma_1>\cdots>\sigma_d$.
Then, since $\tilde{X}=-\tilde{X}$, it follows that $\sigma_0=1$, $\sigma_d=-1$, and more generally, $\sigma_{\ell}=-\sigma_{d-\ell}$ $(0\leqslant \ell\leqslant d)$, i.e., $\theta_{\ell}^*=-\theta_{d-\ell}^*$ $(0\leqslant \ell\leqslant d)$.
To see this, pick any $x,y\in X$ such that $\bm{\xi}_x=-\bm{\xi}_y$.
Then, for every $z\in X$, we have $\bm{\xi}_x\cdot\bm{\xi}_z=\sigma_{\ell}$ if and only if $\bm{\xi}_y\cdot\bm{\xi}_z=-\sigma_{\ell}=\sigma_{d-\ell}$.
For $0\leqslant \ell\leqslant d$, let $k_{\ell}$ denote the valency of the regular graph $(X,R_{\ell})$, where $R_{\ell}\in\mathcal{R}$ corresponds to $A_{\ell}$.
Then, we also have $k_{\ell}=k_{d-\ell}$ $(0\leqslant \ell\leqslant d)$.

We now show that $a_i^*=0$ $(0\leqslant i\leqslant d)$ by induction on $i$.
We always have $a_0^*=0$.
Assume that $a_j^*=0$ $(0\leqslant j<i)$.
By \eqref{3-term recurrence*}, we find recursively that $E_i$ is an even (resp.~odd) $\circ$-polynomial in $E$ if $i$ is even (resp.~odd).
Hence, if we write
\begin{equation*}
    E_i=\frac{1}{|X|}\sum_{\ell=0}^d Q_{\ell,i}A_{\ell},
\end{equation*}
then it follows from \eqref{dual eigenvalues} that $Q_{\ell,i}=(-1)^iQ_{d-\ell,i}$ $(0\leqslant \ell\leqslant d)$.
By these comments and \eqref{dual eigenvalues}, we have
\begin{equation*}
    a_i^*=\frac{1}{|X|m_i}\sum_{\ell=0}^dk_{\ell}\theta_{\ell}^*Q_{\ell,i}^2=\frac{1}{|X|m_i}\sum_{\ell=0}^d k_{d-\ell}(-\theta_{d-\ell}^*)Q_{d-\ell,i}^2=-a_i^*,
\end{equation*}
where the first equality is from \cite[Theorem~2.3.2]{BCN1989B}, which we can easily derive by computing $\operatorname{trace}E_i(E\circ E_i)$ in two ways.
It follows that $a_i^*=0$.
We have now shown that $a_0^*=a_1^*=\cdots=a_d^*=0$, so $E$ is $Q$-bipartite, i.e., \eqref{1st: Q-bipartite} holds.
\end{proof}

It is known (see \cite[Corollary~8.2.2]{BCN1989B}) that every bipartite distance-regular graph with diameter three has precisely two $Q$-polynomial idempotents, except when it is a bipartite Taylor graph (or equivalently, a regular complete bipartite graph minus a perfect matching), in which case it has only one $Q$-polynomial idempotent.
Below is the dual result.

\begin{proposition}\label{tight 5-designs are always P-polynomial}
Referring to Assumption~\ref{assumption: Q only}, assume that $E$ is $Q$-bipartite and $d=3$.
Then, $(X,\mathcal{R})$ is the association scheme of a Taylor graph $\Gamma$.
In particular, $(X,\mathcal{R})$ has precisely two $P$-polynomial adjacency matrices if $\Gamma$ is nonbipartite, and has only one if $\Gamma$ is bipartite.
\end{proposition}

\begin{proof}
By Proposition~\ref{antipodal iff Q-bipartite}, $\tilde{X}$ is antipodal.
We will use the same notation as in the proof of Proposition~\ref{antipodal iff Q-bipartite} with $d=3$.
In particular, $\sigma_0=1$, $\sigma_3=-1$, $\sigma_2=-\sigma_1$, $k_0=k_3=1$, and $k_1=k_2$.
Observe that $0<\sigma_1<1$.
Note also that $A_iA_3=A_{3-i}$ $(0\leqslant i\leqslant 3)$, because for $x,y\in X$, $(A_iA_3)_{x,y}=1$ if and only if $\bm{\xi}_x\cdot(-\bm{\xi}_y)=\sigma_i$, i.e., $\bm{\xi}_x\cdot\bm{\xi}_y=-\sigma_i=\sigma_{3-i}$.
For $x,y\in X$, $(A_1^2)_{x,y}$ counts the common neighbors of $x$ and $y$ in the graph $(X,R_1)$.
Since $0<\arccos\sigma_1<\pi/2$, that $(A_1^2)_{x,y}>0$ implies that $\bm{\xi}_x\cdot\bm{\xi}_y\in\{\sigma_0,\sigma_1,\sigma_2\}$.
It follows that $A_1^2$ is a linear combination of $A_0,A_1$, and $A_2$.
Let us write $A_1^2=kA_0+\lambda A_1+\mu A_2$, where $k=k_1$.
Multiplying both sides by $A_3$ (resp.~$A_3^2=I$), we find $A_1A_2=\mu A_1+\lambda A_2+kA_3$ (resp.~$A_2^2=kA_0+\lambda A_1+\mu A_2$).

Since $kJ=A_1J=A_1(A_0+A_1+A_2+A_3)=k(A_0+A_3)+(\lambda+\mu+1)(A_1+A_2)$, we have $k=\lambda+\mu+1\,(\geqslant 2)$.
If $\mu>0$, then it follows that \eqref{3-term recurrence} holds with $A=A_1$ for $0\leqslant i\leqslant 3$, so $(X,R_1)$ is distance-regular with intersection array $\{k,\mu,1;1,\mu,k\}$, i.e., a Taylor graph.
Likewise, if $\lambda>0$, then $(X,R_2)$ is a Taylor graph with intersection array $\{k,\lambda,1;1,\lambda,k\}$.
Moreover, note that $\mu=0$ if and only if $(X,R_2)$ is a bipartite Taylor graph, and that $\lambda=0$ if and only if $(X,R_1)$ is a bipartite Taylor graph.
The result follows.
\end{proof}

\begin{corollary}\label{Q-bipartite and d=3 imply P-polynomial}
Let $m\geqslant 3$ be an integer.
Then, up to orthogonal transformation, the tight spherical $5$-designs in $S^{m-1}$ are in bijection with the isomorphism classes of nonbipartite Taylor graphs having parameters as in Theorem~\ref{main theorem}\,\eqref{Taylor}.
\end{corollary}

\begin{proof}
Let $\Omega$ be a tight spherical $5$-design in $S^{m-1}$.
We already mentioned that $\Omega$ is antipodal, and is the spherical embedding of a $Q$-polynomial association scheme $(X,\mathcal{R})$ with three classes with respect to a corresponding $Q$-polynomial idempotent $E$.
Observe that $(X,\mathcal{R})$ is uniquely determined by $\Omega$ up to isomorphism (in the natural sense).
By Propositions~\ref{antipodal iff Q-bipartite} and \ref{tight 5-designs are always P-polynomial}, and the last statement in Proposition~\ref{Q-polynomial}, $(X,\mathcal{R}) $ has two $P$-polynomial adjacency matrices, both of which correspond to nonbipartite Taylor graphs.
Now, the result follows from Theorem~\ref{main theorem} and Note~\ref{icosahedron is special}.
\end{proof}

Koolen and Suzuki conjectured that the girth of a distance-regular graph with valency $k\geqslant 3$ is at most twelve; cf.~\cite[Problem~32]{DKT2016EJC}.
Dually, we may consider the following problem:

\begin{problem}
Show that there exists an absolute constant $K>0$ such that we have $t(\tilde{X})\leqslant K$ for the spherical embedding $\tilde{X}$ of every $Q$-polynomial association scheme $(X,\mathcal{R})$ with respect to a corresponding $Q$-polynomial idempotent $E$, provided that the multiplicity $m\geqslant 3$.
Can we take $K=11$?
\end{problem}

\noindent
We may remark that tight spherical $t$-designs in $S^{m-1}$ with $m\geqslant 3$ can exist only for $t\in\{1,2,3,4,5,7,11\}$; cf.~\cite{BD1979JMSJ,BD1980JLMS}.

We conclude the paper with a conjecture on a subclass of $Q$-polynomial association schemes.
To define this subclass, we recall the \emph{Terwilliger algebra} of an association scheme \cite{Terwilliger1992JAC}.
Let $(X,\mathcal{R})$ be an association scheme with $d\geqslant 1$ classes, and let $A_0,A_1,\dots,A_d$ be its adjacency matrices.
Fix a `base vertex' $x\in X$.
For every $i$, let $E_i^*=E_i^*(x)\in\mathbb{R}^{X\times X}$ be the diagonal matrix whose diagonal entries are given by $(E_i^*)_{y,y}=(A_i)_{x,y}$ $(y\in X)$.
Note that $E_0^*+E_1^*+\cdots+E_d^*=I$, and that $E_i^*E_j^*=\delta_{i,j}E_i^*$ $(0\leqslant i,j\leqslant d)$.
The $E_i^*$ are called the \emph{dual idempotents} of $(X,\mathcal{R})$ with respect to $x$.
The Terwilliger algebra of $(X,\mathcal{R})$ with respect to $x$ is the $\mathbb{C}$-subalgebra $\bm{T}(x)$ of $\mathbb{C}^{X\times X}$ generated by the adjacency matrices $A_0,A_1,\dots,A_d$ and the dual idempotents $E_0^*,E_1^*,\dots,E_d^*$.
We note that $\bm{T}(x)$ is semisimple because it is closed under conjugate-transpose.
An irreducible $\bm{T}(x)$-module $W$ is called \emph{thin} (resp.~\emph{dual thin}) if $\dim E_i^*W\leqslant 1$ (resp.~$\dim E_iW\leqslant 1$) for $0\leqslant i\leqslant d$.
The algebra $\bm{T}(x)$ is called \emph{thin} (resp.~\emph{dual thin}) if every irreducible $\bm{T}(x)$-module is thin (resp.~dual thin).
Finally, we say that $(X,\mathcal{R})$ is \emph{thin} (resp.~\emph{dual thin}) if $\bm{T}(x)$ is thin (resp.~dual thin) for every base vertex $x\in X$.
We use the same terminology for distance-regular graphs.
Like the $P$-polynomial and $Q$-polynomial properties, the thin and dual thin properties severely restrict the structure of an association scheme, and have been actively studied; see \cite{DKT2016EJC}.

Collins \cite[Corollary~3.1]{Collins1997GC} showed that if $\Gamma$ is a thin distance-regular graph with diameter $d\geqslant 3$ and valency $k\geqslant 3$, then its girth $g(\Gamma)$ is $3,4,6$, or $8$, and we have $g(\Gamma)=8$ precisely when $\Gamma$ is the collinearity graph of a generalized octagon of order $(1,k-1)$, which has diameter four.
There is a refinement of this result due to Suzuki \cite{Suzuki2006EJC}.
In view of the complete duality between our main result (Theorem~\ref{main theorem}) and Lewis' and Miklavi\v{c}'s results \cite{Lewis2000DM,Miklavic2025EJC} mentioned in Introduction, it seems safe to make the following conjecture:

\begin{conjecture}\label{conjecture}
Referring to Assumption~\ref{assumption: Q only}, assume that $(X,\mathcal{R})$ is dual thin, $d\geqslant 3$, and $m\geqslant 3$.
Then, the strength $t(\tilde{X})$ of $\tilde{X}$ satisfies $t(\tilde{X})\leqslant 7$.
Moreover, if $t(\tilde{X})=7$, then $\tilde{X}$ is a tight spherical $7$-design (so $(X,\mathcal{R})$ is $Q$-bipartite with $d=4$ by Proposition~\ref{antipodal iff Q-bipartite}).
\end{conjecture}

\noindent
By a result of Dickie and Terwilliger \cite[Theorem~2.2]{DT1998JAC}, dual thin $Q$-polynomial association schemes with $a_1^*=0$ are either $Q$-bipartite or almost $Q$-bipartite.
We may also try to dualize (in part) Suzuki's refinement \cite{Suzuki2006EJC}.
Concerning the last statement of Conjecture~\ref{conjecture}, we mention the following, which is dual to \cite[Lemma~2.2]{Collins1997GC}:

\begin{proposition}
Referring to Assumption~\ref{assumption: Q only}, assume that $E$ is $Q$-bipartite and $d\in\{3,4\}$.
Then, $(X,\mathcal{R})$ is dual thin.
\end{proposition}

\begin{proof}
Indeed, we can dualize the entire proof of \cite[Lemma~2.2]{Collins1997GC} as well.

Fix a base vertex $x\in X$.
For $0\leqslant \ell\leqslant d$, let $A_{\ell}^*=A_{\ell}^*(x)\in\mathbb{R}^{X\times X}$ be the diagonal matrix whose diagonal entries are given by $(A_{\ell}^*)_{y,y}=|X|(E_{\ell})_{x,y}$ $(y\in X)$.
We abbreviate $A^*=A_1^*$.
It is known that, for $0\leqslant i,j,\ell\leqslant d$, we have $E_iA_{\ell}^*E_j=0$ if and only if $q_{i,j}^{\ell}=0$; cf.~\cite[Lemma~3.2]{Terwilliger1992JAC}.
We also note that $q_{i,j}^{\ell}=0$ whenever $i+j+\ell$ is odd since $E$ is $Q$-bipartite; see, e.g.,~\cite[Lemma~2.1.1]{Dickie1995D}.
By \cite[Theorem~5.1]{Terwilliger1993JACb}, $\bm{T}(x)$ is dual thin if and only if
\begin{equation}\label{Terwilliger's matrix identities}
    E_iA^*E_jA_{\ell}^*E_i=E_iA_{\ell}^*E_jA^*E_i \qquad (0\leqslant i,j,\ell\leqslant d).
\end{equation}
If $|i-j|\ne 1$, then $q_{i,j}^1=q_{j,i}^1=0$, so that $E_iA^*E_j=E_jA^*E_i=0$ and \eqref{Terwilliger's matrix identities} holds.
Hence, we assume that $j\in\{i+1,i-1\}$.
Now, if $\ell$ is even, then $q_{i,j}^{\ell}=q_{j,i}^{\ell}=0$, so that $E_iA_{\ell}^*E_j=E_jA_{\ell}^*E_i=0$ and \eqref{Terwilliger's matrix identities} holds.
Hence, we also assume that $\ell$ is odd, i.e., $\ell\in\{1,3\}$.
If $\ell=1$, then \eqref{Terwilliger's matrix identities} clearly holds, so we finally consider the case where $\ell=3$.
To this end, we fist show the following:
\begin{equation}\label{intermediate identity 1}
    E_iE_0^*E_{i+1}A^*E_i=b_i^*E_iE_0^*E_i, \qquad E_iE_0^*E_{i-1}A^*E_i=c_i^*E_iE_0^*E_i.
\end{equation}
We prove the left identity above.
The right identity is proved in exactly the same manner.
For $y,z\in X$, we have
\begin{align*}
    (E_iE_0^*E_{i+1}A^*E_i)_{y,z} &= \sum_{u,v\in X}(E_i)_{y,u}(E_0^*)_{u,u}(E_{i+1})_{u,v}(A^*)_{v,v}(E_i)_{v,z} \\
    &= |X|(E_i)_{y,x}\sum_{v\in X} (E_{i+1})_{x,v}E_{x,v}(E_i)_{v,z} \\
    &= |X|(E_i)_{y,x}\sum_{v\in X}(E_{i+1}\circ E)_{x,v}(E_i)_{v,z} \\
    &= (E_i)_{y,x}\sum_{v\in X} (b_i^*E_i+c_{i+2}^*E_{i+2})_{x,v}(E_i)_{v,z} \\
    &= b_i^*(E_i)_{y,x}(E_i)_{x,z} \\
    &= b_i^*\sum_{u\in X} (E_i)_{y,u}(E_0^*)_{u,u}(E_i)_{u,z} \\
    &= b_i^*(E_iE_0^*E_i)_{y,z},
\end{align*}
where we have used \eqref{3-term recurrence*}.
This proves the identity.
Note that the RHS in each of the identities in \eqref{intermediate identity 1} is symmetric, from which it follows that
\begin{equation}\label{intermediate identity 2}
    E_iE_0^*E_jA^*E_i=E_iA^*E_jE_0^*E_i.
\end{equation}
By $E_0+E_1+\cdots+E_d=I=A_0$, we have $A_0^*+A_1^*+\cdots+A_d^*=|X|E_0^*$.
By this and since $E_iA_h^*E_j=E_jA_h^*E_i=0$ for even $h$, we have
\begin{equation}\label{intermediate identity 3}
    E_iE_0^*E_j=\frac{1}{|X|}(E_iA^*E_j+E_iA_3^*E_j), \quad E_jE_0^*E_i=\frac{1}{|X|}(E_jA^*E_i+E_jA_3^*E_i).
\end{equation}
Combining \eqref{intermediate identity 2} and \eqref{intermediate identity 3}, we obtain \eqref{Terwilliger's matrix identities} for $j\in\{i+1,i-1\}$ and $\ell=3$.
It follows that $\bm{T}(x)$ is dual thin.
Since $x$ is arbitrary, $(X,\mathcal{R})$ is dual thin, as desired.
\end{proof}

\section*{Acknowledgments}
JL was a JSPS International Research Fellow at Tohoku University at the time of this work.
JL, AM, SS, and HT were respectively supported by JSPS KAKENHI Grant Numbers \href{https://kaken.nii.ac.jp/en/grant/KAKENHI-PROJECT-24KF0099/}{JP24KF0099}, \href{https://kaken.nii.ac.jp/en/grant/KAKENHI-PROJECT-25K07095/}{JP25K07095}, \href{https://kaken.nii.ac.jp/en/grant/KAKENHI-PROJECT-22K03410/}{JP22K03410}, and \href{https://kaken.nii.ac.jp/en/grant/KAKENHI-PROJECT-23K03064/}{JP23K03064}.
This work was also partially supported by the Research Institute for Mathematical
Sciences, an International Joint Usage/Research Center located in Kyoto
University.

\appendix

\section{SageMath code}
\label{sec: SageMath codes}

Below is the SageMath code we used to find the polynomial $\tilde{h}_2(q)$ in Lemma~\ref{less than 2}:

\medskip

\begin{verbatim}
R.<s,q> = PolynomialRing(QQ)  # s:=s^*
phi2 = ((s*q^2-1)*(s*q^3-1)*(s*q^6-1)*(s^2*q^12-1)
        -2*(s*q^4-1)*(s*q^5-1)*(s*q^4-1)*(s^2*q^10-1))
psi = q^8*(2*q+1)*s^2-q^2*(q^5-1)*s-q-2
I = ideal(phi2,psi)
J = I.elimination_ideal(s)
th2 = J.gens()[0]
th2.factor()
\end{verbatim}

\medskip
\noindent
The polynomials $\tilde{h}_1(q),\tilde{h}_{10}(q)$, and $\tilde{h}_9(q)$ in Lemmas~\ref{L2 unequal to 1}, \ref{less than 10}, and \ref{greater than 9} are computed similarly.

\bigskip
\bigskip

\end{document}